\documentclass[11pt,a4paper,reqno]{amsart}
\usepackage[T1]{fontenc}
\usepackage[latin9]{inputenc}
\usepackage{color}
\usepackage{amsthm}
\usepackage{varioref}
\usepackage{amsbsy}
\usepackage{amstext}
\usepackage{amssymb}
\usepackage{esint}
\usepackage[unicode=true,pdfusetitle,
 bookmarks=true,bookmarksnumbered=false,bookmarksopen=false,
 breaklinks=false,pdfborder={0 0 0},backref=false,colorlinks=true]
 {hyperref}
\usepackage{breakurl}

\makeatletter

\numberwithin{equation}{section}
\numberwithin{figure}{section}
\theoremstyle{plain}
\newtheorem{thm}{\protect\theoremname}[section]
  \theoremstyle{remark}
  \newtheorem{rem}[thm]{\protect\remarkname}
  \theoremstyle{definition}
  \newtheorem{example}[thm]{\protect\examplename}
  \theoremstyle{plain}
  \newtheorem{cor}[thm]{\protect\corollaryname}
  \theoremstyle{remark}
  \newtheorem*{acknowledgement*}{\protect\acknowledgementname}
  \theoremstyle{plain}
  \newtheorem{prop}[thm]{\protect\propositionname}
  \theoremstyle{definition}
  \newtheorem{defn}[thm]{\protect\definitionname}
  \theoremstyle{plain}
  \newtheorem{lem}[thm]{\protect\lemmaname}

\usepackage{a4wide}

\newcommand{\Z}{\mathbb{Z}}
\newcommand{\R}{\mathbb{R}}

\newcommand{\id}{\mathrm{id}}

\newcommand{\beq}{\begin{equation}}
\newcommand{\beqn}{\begin{equation}\nonumber}
\newcommand{\eeq}{\end{equation}}

\newcommand{\bea}{\begin{equation}\begin{aligned}}
\newcommand{\bean}{\begin{equation}\begin{aligned}\nonumber}
\newcommand{\eea}{\end{aligned}\end{equation}}

\makeatother

  \providecommand{\acknowledgementname}{Acknowledgement}
  \providecommand{\corollaryname}{Corollary}
  \providecommand{\definitionname}{Definition}
  \providecommand{\examplename}{Example}
  \providecommand{\lemmaname}{Lemma}
  \providecommand{\propositionname}{Proposition}
  \providecommand{\remarkname}{Remark}
\providecommand{\theoremname}{Theorem}

\begin{document}

\title{Orderability and the Weinstein Conjecture}

\author{Peter Albers}

\author{Urs Fuchs}

\author{Will J.~Merry}

\address{ Peter Albers\\
 Mathematisches Institut\\
 Westf\"alische Wilhelms-Universit\"at M\"unster}

\email{peter.albers@wwu.de}

\address{ Urs Fuchs\\
 Mathematisches Institut\\
 Westf\"alische Wilhelms-Universit\"at M\"unster}

\email{ufuchs@wwu.de}

\address{Will J.~Merry\\
 Department of Mathematics\\
 ETH Z\"urich}

\email{merry@math.ethz.ch}

\begin{abstract}
In this article we prove that the Weinstein conjecture holds for contact
manifolds $(\Sigma,\xi)$ for which $\mbox{Cont}_{0}(\Sigma,\xi)$
is non-orderable in the sense of Eliashberg-Polterovich \cite{EliashbergPolterovich2000}. More precisely, we establish a link between orderable and hypertight contact manifolds. In addition, we prove for certain contact manifolds a conjecture by Sandon \cite{Sandon2013} on the existence of translated points in
the non-degenerate case. 
\end{abstract}

%\date{\today}

\maketitle

\section{Introduction}

One of the driving questions in the field of contact geometry is the
famous \textit{Weinstein conjecture} \cite{Weinstein1979} which asserts
for a closed coorientable contact manifold $(\Sigma,\xi)$ that any supporting contact form admits a periodic Reeb orbit. See for
instance \cite{Hutchings2010} for more information.

In \cite{EliashbergPolterovich2000} Eliashberg and Polterovich introduced
the concept of \textit{orderability} of contact manifolds, which is
closely related to the question of contact (non-)squeezing, see \cite{EliashbergKimPolterovich2006}.
We denote by $\mbox{Cont}_{0}(\Sigma,\xi)$ the group of contactomorphisms
of $(\Sigma,\xi)$ which are contact isotopic to the identity, and
by $\widetilde{\mbox{Cont}}_{0}(\Sigma,\xi)$ its universal cover.
Eliashberg--Polterovich proved that $\mbox{Cont}_{0}(\Sigma,\xi)$
is non-orderable if and only if there exists a \textit{positive} loop
$\varphi$ in $\mbox{Cont}_{0}(\Sigma,\xi)$, and similarly that $\widetilde{\mbox{Cont}}_{0}(\Sigma,\xi)$
is non-orderable if and only if there exists a \emph{positive contractible
}loop $\varphi$, see Section \ref{sec:Preliminaries} for details. Here is our first result.
\begin{thm}
\label{thm:main} The Weinstein conjecture holds for any contact manifold
$(\Sigma,\xi)$ for which $\mbox{\emph{Cont}}_{0}(\Sigma,\xi)$ is
non-orderable. If in addition $\widetilde{\mbox{\emph{Cont}}}_{0}(\Sigma,\xi)$
is non-orderable then every supporting contact form admits a \emph{contractible}
closed Reeb orbit. \end{thm}
\begin{rem}
Note that the original notion of orderability in \cite{EliashbergPolterovich2000}
actually concerns the universal cover $\widetilde{\mbox{Cont}}_{0}(\Sigma,\xi)$
which is much more restrictive than orderability of $\mathrm{Cont}_0(\Sigma,\xi)$. There are many examples of contact manifolds for which $\mbox{Cont}_{0}(\Sigma,\xi)$ is non-orderable
while $\widetilde{\mbox{Cont}}_{0}(\Sigma,\xi)$ is orderable, e.g.~$\mathbb{R}\mbox{P}^{2n-1}$. 
\end{rem}

\begin{rem}
\label{rem:u phi}
Given a loop $\varphi=\{\varphi_{t}\}_{t\in S^{1}}$ of contactomorphisms
we denote by $u_{\varphi}\in\pi_{1}(\Sigma)$ the homotopy class of
the loop $t\mapsto\varphi_{t}(x)$, and by $\tilde{u}_{\varphi}$
the corresponding free homotopy class (i.e.~the image of $u_{\varphi}$
under the map $\pi_{1}\rightarrow\pi_{1}/\mbox{conjugacy}=[S^{1},\Sigma]$). 
Theorem \ref{thm:main} can be sharpened as follows: if there exists
a positive loop $\varphi$ in $\mbox{Cont}_{0}(\Sigma,\xi)$, then
for any supporting contact form, either there exists
a closed contractible Reeb orbit or there exists a closed Reeb orbit
in the free homotopy class $\tilde{u}_{\varphi}$. In fact, the same assertion is true for any loop $\varphi$ with spectral number $c(\varphi)\neq0$, see Definition \ref{def of c} and the proof of Theorem \ref{thm:main} \vpageref{proof of theorem main}.
\end{rem}
\begin{example}
\label{exa:prequant}
For all contact manifolds $(\Sigma,\xi)$ admitting a supporting contact
form with periodic Reeb flow,  $\mbox{Cont}_{0}(\Sigma,\xi)$ is non-orderable, since in this case the Reeb flow is itself a positive loop (cf. Section \ref{sec:Preliminaries}). An interesting class of examples of contact manifolds are \emph{prequantization spaces}. Here one begins with a closed symplectic manifold $(M,\omega)$ for which the de Rham cohomology class
$[\omega]$ has a primitive integral lift in $\mbox{H}^{2}(M;\mathbb{Z})$.
Consider a circle bundle $p:\Sigma_k \rightarrow M$ with Euler class
$k[\omega]$ for some $k \in \mathbb{Z}$ with $k \ne 0$, and connection 1-form $\alpha$ with $p^{*}(k\omega)=-d\alpha$.
Then $(\Sigma_k,\alpha)$ is a contact manifold whose associated Reeb
flow is periodic. The closed Reeb orbits are the fibres of the bundle.
The long exact homotopy sequence of the fibration is
\begin{equation}
\pi_{2}(M)\overset{q_k}{\rightarrow}\pi_{1}(S^{1})\rightarrow\pi_{1}(\Sigma_k)\rightarrow\pi_{1}(M)\rightarrow0.\label{eq:the map q}
\end{equation}
The map $q_k$ is non-trivial if and only if the homotopy class of the fibre
is torsion (and note if $q_k$ is non-trivial then so is $q_{nk}$ for all $n \ne 0$). 
See Example \ref{exa:Prequantization-spaces-for} below for the relevance of this last statement.
\end{example}
A contact manifold $(\Sigma,\xi)$ is called \textit{hypertight} if
it admits a supporting contact form without any contractible Reeb
orbits, see for example \cite{ColinHonda2005} for a construction
of hypertight contact manifolds. The 3-torus $\mathbb{T}^{3}$ (equipped with
any one of the standard contact structures $\alpha_{k}=\cos(2\pi kr)ds+\sin(2\pi kr)dt$)
is a familiar example. We point out that if a contact manifold is \emph{not} hypertight
then all supporting contact forms possess \emph{contractible} Reeb
orbits, which is a stronger assertion than the Weinstein Conjecture! 
\begin{thm}
\label{thm:main2} If $(\Sigma,\xi)$ is hypertight then for any positive
loop $\varphi$ in $\mbox{\emph{Cont}}_{0}(\Sigma,\xi)$ the class $u_{\varphi}\in\pi_{1}(\Sigma)$ is of
infinite order.
\end{thm}
The following corollary is immediate from Theorem \ref{thm:main2}.
\begin{cor}
\label{cor:inf_order}
If $(\Sigma,\xi)$ is hypertight then all positive
loops of contactomorphisms are of infinite order in $\pi_{1}(\mbox{\emph{Cont}}_{0}(\Sigma,\xi))$. 
\end{cor}
\begin{rem}
Theorem \ref{thm:main2} and Corollary \ref{cor:inf_order} illustrate the sharp contrast with life in the symplectic world. Indeed, if $(M, \omega)$ is a compact symplectic manifold then the evaluation map $ \pi_1( \mathrm{Ham}(M, \omega)) \to \pi_1(M)$ is always trivial, cf. \cite[Exercise 11.28]{McDuffSalamon1998}.
\end{rem}

Using Theorem \ref{thm:main} and Theorem \ref{thm:main2}, we can now improve Example \ref{exa:prequant} to obtain:
\begin{example}
\label{exa:Prequantization-spaces-for}Prequantization spaces $(\Sigma,\xi)$ are hypertight if and only if the fibre is not torsion. $\mbox{Cont}_{0}(\Sigma,\xi)$ is always non-orderable. If in addition $\widetilde{\mbox{Cont}}_{0}(\Sigma,\xi)$ is non-orderable then the homotopy class of the fibre is torsion. 
\end{example}
\begin{rem}
Example \ref{exa:Prequantization-spaces-for} is sharp in the following
sense: $\mathbb{R}\mbox{P}^{2n-1}$ is a prequantization space with torsion fibres, but $\widetilde{\mathrm{Cont}}_0(\mathbb{R}\mbox{P}^{2n-1},\xi_{\mathrm{st}})$ is orderable. This can proved using Givental's nonlinear Maslov index, cf. \cite{EliashbergPolterovich2000}. Otto van Koert explained to us that Example \ref{exa:Prequantization-spaces-for} can also be shown by using contact homology.
\end{rem}

Finally, we prove for certain contact manifolds a conjecture by
Sandon \cite[Conjecture 1.2]{Sandon2013} in the non-degenerate case.
For this we recall that if $\psi\in\mbox{Cont}_{0}(\Sigma,\xi)$ and
$\alpha$ is a supporting contact form then a point $x\in\Sigma$
is a \emph{translated point }of $\psi$ (with respect to $\alpha$)
if $\psi(x)$ belongs to the same Reeb orbit as $x$ does, and if $\psi$
is ``exact at $x$'' in the sense that $\psi^{*}\alpha|_{x}=\alpha|_{x}$.
\begin{thm}
\label{thm:margherita} Let $(\Sigma,\xi=\ker\,\alpha)$ be a contact
manifold such that $\alpha$ has no contractible
closed Reeb orbits. Then every non-degenerate $\psi\in\mbox{\emph{Cont}}_{0}(\Sigma,\xi)$
has at least $\sum_{j=0}^{\dim\,\Sigma}\dim\,\mbox{\emph{H}}_{j}(\Sigma;\mathbb{Z}_{2})$ many translated points.\end{thm}
\begin{rem}
The non-degeneracy hypothesis in Theorem \ref{thm:margherita} is a
standard one, and is satisfied generically. See Definition \ref{def: nondegen}
below. 
\end{rem}
\begin{rem}
Sandon \cite{Sandon2011b} was the first to discover a connection between translated points and orderability and other contact rigidity phenomena.  She informed us that she is working on a Floer-theoretical approach \cite{Sandon2014} and that Z\'ena\"idi is working on an approach based on Legendrian Contact Homology \cite{Zenaidi2014}. We expect interesting interactions between this paper and the approaches followed by Sandon and Z\'ena\"idi.
\end{rem}
\begin{acknowledgement*}
We are grateful to Joel Fish for pointing out to us an alternative proof of Theorem \ref{thm: the actual result} using his method of target local compactness \cite{Fish2011}. See Remark \ref{rem:TLC}. We are also grateful to Paul Biran, Otto van Koert and Leonid Polterovich for their helpful comments and discussions. PA and UF are supported by the SFB 878 - Groups, Geometry and Actions. WM is supported
by an ETH Postdoctoral Fellowship. 
\end{acknowledgement*}

\section{\label{sec:Preliminaries}Preliminaries}

We denote by $\mbox{Cont}_{0}(\Sigma,\xi)$ the identity component
of the group of contactomorphisms. Unless specified otherwise a \emph{path
}$\varphi=\{\varphi_{t}\}_{0\leq t\leq1}$ of contactomorphisms is
always smoothly parametrized and begins at the identity. We denote
by $\mathcal{P}\mbox{Cont}_{0}(\Sigma,\xi)$ the set of all such paths.
The universal cover $\widetilde{\mbox{Cont}}_{0}(\Sigma,\xi)$ is
then $\mathcal{P}\mbox{Cont}_{0}(\Sigma,\xi)/\sim$, where $\sim$
denotes the equivalence relation of being homotopic with fixed endpoints.
Suppose $\alpha\in\Omega^{1}(\Sigma)$ is a contact form defining
$\xi$. To a path $\varphi=\{\varphi_{t}\}_{0\leq t\leq1}\in\mathcal{P}\mbox{Cont}_{0}(\Sigma,\xi)$
we can uniquely associate its contact Hamiltonian $h_{t}$ defined
by
\begin{equation}
h_{t}\circ\varphi_{t}:=\alpha\left(\frac{d}{dt}\varphi_{t}\right):\Sigma\rightarrow\mathbb{R},
\end{equation}
see for instance \cite[Chapter 2.3]{Geiges2008}. 
The following definitions are taken from \cite{EliashbergPolterovich2000}. Given $\psi\in\mbox{Cont}_{0}(\Sigma,\xi)$
let us say that 
\begin{equation}
\psi\geq\mbox{id}
\end{equation}
if there exists a path $\varphi=\{\varphi_{t}\}_{0\leq t\leq1}$ with
$\psi=\varphi_{1}$ such that the contact Hamiltonian $h_{t}$ of
$\varphi$ is everywhere non-negative. We say that $\mbox{Cont}_{0}(\Sigma,\xi)$
is \emph{orderable }if the relation $\geq$ induces a partial order
on $\mbox{Cont}_{0}(\Sigma,\xi)$. Otherwise $\mbox{Cont}_{0}(\Sigma,\xi)$
is \emph{non-orderable}. We can play the same game with $\widetilde{\mbox{Cont}}_{0}(\Sigma,\xi)$:
given $\Phi\in\widetilde{\mbox{Cont}}_{0}(\Sigma,\xi)$ let us say
that 
\begin{equation}
\Phi\geq_{u}\mbox{id}
\end{equation}
if there exists a representative $\{\varphi_{t}\}_{0\leq t\leq1}\in\mathcal{P}\mbox{Cont}_{0}(\Sigma,\xi)$
of $\Phi$ whose contact Hamiltonian is everywhere non-negative. Then
we say that $\widetilde{\mbox{Cont}}_{0}(\Sigma,\xi)$ is \emph{orderable
}if $\geq_{u}$ induces a partial order on $\widetilde{\mbox{Cont}}_{0}(\Sigma,\xi)$;
otherwise $\widetilde{\mbox{Cont}}_{0}(\Sigma,\xi)$ is \emph{non-orderable}. 
\begin{rem}
We call a loop of contactomorphisms whose contact Hamiltonian is everywhere \emph{strictly}
positive a ``positive loop'' (of contactomorphisms.) 
\end{rem}

Although we used a supporting contact
form $\alpha$ to define the notion of positivity, it is easy to see
that the positivity of a path $\{\varphi_{t}\}_{0\leq t\leq1}$ is
equivalent to the vector field $\frac{d}{dt}\varphi_{t}$ defining
the given coorientation of $\xi$, and this of course does not depend
on the choice of $\alpha$.
The following characterization of orderability by Eliashberg-Polterovich
\cite[Section 2.1]{EliashbergPolterovich2000} is crucial.
\begin{prop}[{\cite[Section 2.1]{EliashbergPolterovich2000}}]
\label{prop:EP_criterion}$\mbox{\emph{Cont}}_{0}(\Sigma,\xi)$ is
non-orderable if and only if there exists a positive loop of contactomorphisms. Moreover $\widetilde{\mbox{\emph{Cont}}}_{0}(\Sigma,\xi)$
is non-orderable if and only if there exists a positive \emph{contractible} loop.\end{prop}

\section{\label{sec:Rabinowitz-Floer-homology}Rabinowitz Floer homology}
\begin{defn}
Let us say that a contact form $\alpha$ supporting $\xi$ is WCRO if $\alpha$ is \emph{without contractible Reeb orbits}. Thus hypertight
contact manifolds are exactly those which admit a WCRO contact form.
\end{defn}
The aim of this section is to define the \emph{Rabinowitz Floer homology
}$\mbox{RFH}_{*}(\Sigma,\alpha;\varphi)$ associated to a contact
manifold $(\Sigma,\xi)$ with a supporting WCRO contact
form $\alpha$ and a path $\varphi=\{\varphi_{t}\}_{0\leq t\leq1}$
in $\mathcal{P}\mbox{Cont}_{0}(\Sigma,\xi)$. The main novelty in
our treatment is that we do \emph{not} need a filling of the contact
manifold. Rabinowitz Floer homology was first constructed in \cite{CieliebakFrauenfelder2009}
by Cieliebak-Frauenfelder. The construction we present now is derived
from \cite{AlbersMerry2013a}. We start with some notation.
\begin{rem}
\label{rem:inverse para}By convention, in what follows a Reeb orbit
of period $T<0$ is by definition the inverse parametrization of a
Reeb orbit of period $-T$.
\end{rem}
From now on we fix a WCRO contact form $\alpha$ defining $\xi$,
and denote by $R$ its Reeb vector field and $\theta^{t}:\Sigma\rightarrow\Sigma$
the Reeb flow. We emphasize that we are \emph{not} assuming that the
non-contractible Reeb orbits of $\alpha$ are non-degenerate (cf.
Remark \ref{rem: the functional is Morse Bott} below). The \emph{symplectization} $(S\Sigma,d\lambda)$ of $(\Sigma,\alpha)$ is the manifold $S\Sigma:=(0,\infty)\times\Sigma$
equipped with the symplectic form $d\lambda$, where $\lambda:=r\alpha$.
Here $r$ is the coordinate on $(0,\infty)$. 
By a common abuse of notation we identify $R$ with the vector field
$(0,R)$ on $S\Sigma$. A path $\{\varphi_{t}\}_{0\leq t\leq1}\in\mathcal{P}\mbox{Cont}_{0}(\Sigma,\xi)$
can be lifted to a Hamiltonian isotopy on $S\Sigma$ as follows. Write
$\varphi_{t}^{*}\alpha=\rho_{t}\varphi_{t}$ for $\rho_{t}:\Sigma\rightarrow(0,\infty)$.
Then 
\begin{equation}
(x,r)\mapsto\left(\frac{r}{\rho_{t}(x)},\varphi_{t}(x)\right)\label{eq:lift path to symplectization}
\end{equation}
is the Hamiltonian flow of 
\begin{equation}
H_{t}(x,r):=rh_{t}(x):S\Sigma\rightarrow\mathbb{R}.\label{eq:Ht}
\end{equation}
Recall from the Introduction the definition of a translated point
of a contactomorphism. This notion was introduced by Sandon in \cite{Sandon2012}.
For us their relevance is that the generators of the Rabinowitz Floer
homology $\mbox{RFH}_{*}(\Sigma,\alpha;\varphi)$ are precisely
the translated points of $\varphi$. 
\begin{defn}
\label{def trans poin}Fix $\psi\in\mbox{Cont}_{0}(\Sigma,\xi)$,
and write $\psi^{*}\alpha=\rho\alpha$ for a smooth positive function
$\rho$ on $\Sigma$. A \emph{translated point }of $\psi$ is a point
$x\in\Sigma$ with the property that there exists $\eta\in\mathbb{R}$
such that 
\begin{equation}
\psi(x)=\theta^{\eta}(x),\ \ \ \mbox{and}\ \ \ \rho(x)=1.\label{eq:TP}
\end{equation}
We call $\eta$ a \emph{time-shift }of $x$. Note that if the leaf
$\{\theta^{t}(x)\}_{t\in\mathbb{R}}$ is closed then a time-shift
$\eta$ is not uniquely determined by $x$. 
\end{defn}

In order to define the (perturbed) Rabinowitz action functional we will work with
a collection of cutoff functions depending on parameters $\kappa>\kappa'>0$
of the following form. Let $m_{\kappa,\kappa'}:\mathbb{R}\rightarrow\mathbb{R}$
satisfy 
\begin{equation}
m_{\kappa,\kappa'}(r)=\begin{cases}
r-1, & r\in[e^{-\kappa'},e^{\kappa'}],\\
c_{1}, & r\in[e^{\kappa},\infty),\\
c_{2}, & r\in(0,e^{-\kappa}],
\end{cases}\ \ \ m_{ \kappa, \kappa'}'(r)\geq0,\label{eq:m kappa kappa prime}
\end{equation}
for some suitable constants $c_{1},c_{2}\in\mathbb{R}$. Similarly
let $\varepsilon_{\kappa,\kappa'}\in C^{\infty}((0,\infty),[0,1])$
denote a smooth function such that 
\begin{equation}
\varepsilon_{\kappa,\kappa'}(r)=\begin{cases}
1, & r\in[e^{-\kappa'},e^{\kappa'}],\\
0, & r\in(0,e^{-\kappa}]\cup[e^{\kappa},\infty).
\end{cases}\label{eqn:epsilon_function}
\end{equation}
Next, fix a smooth function $\nu:S^{1}\rightarrow\mathbb{R}$ with
\begin{equation}
\nu(t)=0\ \forall t\in[\tfrac{1}{2},1],\ \ \ \mbox{and}\ \ \ \int_{0}^{1}\nu(t)dt=1,
\end{equation}
and fix a smooth monotone map $\chi:[0,1]\rightarrow[0,1]$ with $\chi(\tfrac{1}{2})=0$
and $\chi(1)=1$. Denote by $\mathcal{L}S\Sigma$ the space of \emph{contractible} smooth loops $u:S^1 \rightarrow S \Sigma$.
The perturbed Rabinowitz action functional will be a functional defined
on $\mathcal{L}S\Sigma\times\mathbb{R}$.
\begin{defn}
Fix a path $\varphi=\{\varphi_{t}\}_{0\leq t\leq1}\in\mathcal{P}\mbox{Cont}_{0}(\Sigma,\xi)$.
Let $h_{t}:\Sigma\rightarrow\mathbb{R}$ denote the contact Hamiltonian
and set $H_{t}=rh_{t}:S\Sigma\rightarrow\mathbb{R}$. We define the
\emph{perturbed Rabinowitz action functional }associated to $\varphi$
and a pair of real numbers $\kappa>\kappa'>0$, written
\begin{equation}
\mathcal{A}_{\varphi}^{\kappa,\kappa'}:\mathcal{L}S\Sigma\times\mathbb{R}\rightarrow\mathbb{R},
\end{equation}
as follows: Given $u\in\mathcal{L}S\Sigma$ write $u=(a,f)$, so that
$a:S^{1}\to(0,\infty)$ and $f\in\mathcal{L}\Sigma$. Set 
\begin{equation}
\mathcal{A}_{\varphi}^{\kappa,\kappa'}(u,\eta):=\int_{0}^{1}u^{*}\lambda-\eta\int_{0}^{1}\nu(t)m_{\kappa,\kappa'}(a(t))dt-\int_{0}^{1}\dot{\chi}(t)\varepsilon_{\kappa,\kappa'}(a(t))H_{\chi(t)}(u(t))dt.\label{eq:Rab functional}
\end{equation}

\end{defn}
In \cite[Proposition 2.5]{AlbersMerry2013a} we proved:
\begin{lem}
\label{prop:linfinity-1}Assume $\varphi=\{\varphi_{t}\}_{0 \le t \le 1}$ is a path of contactomorphims. Write $\varphi_t^*\alpha = \rho_t \alpha$, for smooth functions $\rho_{t}:\Sigma\to(0,\infty)$. Set
\begin{equation}\label{eq:def_kappa_varphi}
\kappa(\varphi):=\max_{t\in[0,1]}\left|\int_{0}^{t}\max_{x\in\Sigma}\frac{\dot{\rho}_{\tau}(x)}{\rho_{\tau}(x)^{2}}d\tau\right|
\end{equation}
If $\kappa>\kappa'>\kappa(\varphi)$ and $(u=(a,f),\eta)\in\mbox{\emph{Crit}}(\mathcal{A}_{\varphi}^{\kappa,\kappa'})$
then $a(S^{1})\subseteq(e^{-\kappa'},e^{\kappa'})$. Thus for $\kappa>\kappa'>\kappa(\varphi)$,
a pair $(u,\eta)$ is a critical point of $\mathcal{A}_{\varphi}^{\kappa,\kappa'}$
only if $x:=f(\tfrac{1}{2})$ is a translated point of $\varphi_1$,
with $-\eta$ a time-shift of $x$. Conversely every such pair $( x, \eta)$ gives rise to a unique critical point of $ \mathcal{A}_{ \varphi}^{ \kappa , \kappa'}$. Moreover one has 
\begin{equation}
\mathcal{A}_{\varphi}^{\kappa,\kappa'}(u,\eta)=\eta.\label{eq:value on crit}
\end{equation}

\end{lem}
In what follows we tacitly assume that $\kappa>\kappa(\varphi)$,
even if this is not explicitly stated. In order to simply the notation
we define 
\begin{equation}
\mathcal{A}_{\varphi}^{\kappa}:=\mathcal{A}_{\varphi}^{\kappa,\kappa'},
\end{equation}
where $\kappa'$ is any number such that $\kappa>\kappa'>\kappa(\varphi)$.
The precise choice of $\kappa'$ is unimportant in all of what follows.
Next, we set 
\begin{equation}
\mbox{Spec}(\varphi):=\mathcal{A}_{\varphi}^{\kappa}(\mbox{Crit}(\mathcal{A}_{\varphi}^{\kappa})).
\end{equation}

\begin{defn}
\label{def: nondegen}A path $\varphi$ is \emph{non-degenerate }if
$\mathcal{A}_{\varphi}^{\kappa}:\mathcal{L}S\Sigma\times\mathbb{R}\rightarrow\mathbb{R}$
is a Morse-Bott function for some (and hence any) $\kappa>\kappa(\varphi)$. 
\end{defn}

\begin{rem}\label{rem: the functional is Morse Bott}
In \cite{AlbersMerry2013a} we explained why a generic path $\varphi$
is non-degenerate (actually a stronger result is true: for generic
$\varphi$ the functional $\mathcal{A}_{\varphi}^{\kappa}$ is even Morse). 
Moreover $\mathrm{Spec}(\varphi)$ is always a nowhere dense subset of $\mathbb{R}$ (even in the degenerate case), cf. \cite[Lemma 3.8]{Schwarz2000}. Finally given $\varphi=\{\varphi_t\} \in \mathcal{P}\mbox{Cont}_{0}(\Sigma,\xi)$ the set $\mathrm{Spec}(\varphi)$ only depends on the endpoint $\varphi_1$ of the path $\varphi$. This follows from Lemma \ref{prop:linfinity-1} together with the fact that the definition of a translated point only involves the map $\varphi_1$. 

The non-perturbed Rabinowitz
action functional $\mathcal{A}_{\textrm{id}}^{\kappa}$ is Morse-Bott. Indeed,
since $\alpha$ is WCRO, the critical points of $\mathcal{A}_{\textrm{id}}^{\kappa}$
are all of the form $(u=(a,f),\eta=0)$ with $a(t)\equiv1$ and $f(t)\equiv x$
for some point $x\in\Sigma$. This is a Morse-Bott component, cf.
\cite[Lemma 2.12]{AlbersFrauenfelder2010c}. 
\end{rem}
Fix a family $J=\{J_{t}\}_{t\in S^{1}}$ of almost complex structures
on $S\Sigma$ compatible with $d\lambda$. Here we use the (slightly unusual) sign convention that compatibility means that $d\lambda(J_{t}\cdot,\cdot)$ defines a family of 
Riemannian metrics on $S\Sigma$. We assume in addition that $J$ is
\emph{SFT-like} and independent of $t$ outside a compact set. Here
an almost complex structure on $S\Sigma$ is SFT-like if it is invariant
under the translations $(r,x)\mapsto(e^{c}r,x)$ for $c\in\mathbb{R}$,
and if it preserves $\xi$ and satisfies $JR=r\partial_{r}$. We denote
by $\left\langle \left\langle \cdot,\cdot\right\rangle \right\rangle _{J}$
the $L^{2}$-inner product defined by
\begin{equation}
\left\langle \left\langle (\hat{u},\hat{\eta}),(\hat{v},\hat{\tau})\right\rangle \right\rangle _{J}:=\int_{S^{1}}d\lambda(J_{t}\hat{u},\hat{v})dt+\hat{\eta}\hat{\tau}\ \ \ \mbox{for }(\hat{u},\hat{\eta}),(\hat{v},\hat{\tau})\in T_{(u,\eta)}(\mathcal{L}S\Sigma\times\mathbb{R}).
\end{equation}
We denote by $\nabla_{J}\mathcal{A}_{\varphi}^{\kappa}$ the gradient
of $\mathcal{A}_{\varphi}^{\kappa}$ with respect to the inner
product $\left\langle \left\langle \cdot,\cdot\right\rangle \right\rangle _{J}$
on $\mathcal{L}S\Sigma\times\mathbb{R}$, that is, the integro-differential
operator 
\begin{equation}
\nabla_{J}\mathcal{A}_{\varphi}^{\kappa}(u,\eta)=\left(J_{t}(u)\left(\partial_{t}u-\eta\nu X_{m_{ \kappa, \kappa'}}(u)-\varepsilon_{ \kappa , \kappa'} \dot{\chi}X_{H_{\chi}}(u)\right),-\int_{S^{1}}\nu m_{ \kappa, \kappa'}(u)dt\right).\label{eq:gradientA}
\end{equation}
Negative gradient flow lines of $\nabla_{J}\mathcal{A}_{\varphi}^{\kappa}$
solve the equation 
\begin{equation}
\partial_{s}(u,\eta)+\nabla_{J}\mathcal{A}_{\varphi}^{\kappa}(u,\eta)=0,
\end{equation}
and have energy 
\begin{equation}
\mathbb{E}_{J}(u,\eta):=\int_{-\infty}^{\infty}\int_{0}^{1}|\partial_{s}(u,\eta)|_{J}^{2}dtds\;.
\end{equation}
\begin{defn}
Denote by $\mathcal{M}(\varphi,\kappa,J)$ the set of all such flow
lines with finite energy $\mathbb{E}_{J}(u,\eta)< e^{ - \kappa} \wp ( \alpha)$, where $\wp ( \alpha)$ denotes the minimal period of a closed contractible Reeb orbit of $ \alpha$.
\end{defn}
 Of course the assumption that $ \alpha$ is WCRO implies that $ \wp ( \alpha) = + \infty$, and hence in this paper $\mathcal{M}(\varphi,\kappa,J)$ simply denotes the set of \emph{all} finite energy flow lines. Nevertheless, we state (and prove) Theorem \ref{thm:compactness} below without the assumption that $ \alpha$ is WCRO, as this will be useful in a forthcoming paper. If $\mathcal{A}_{\varphi}^{\kappa}$
is Morse-Bott then any element $(u,\eta)\in\mathcal{M}(\varphi,\kappa,J)$
converges to critical points $(u_{\pm},\eta_{\pm})\in\mbox{Crit}(\mathcal{A}_{\varphi}^{\kappa})$
as $s\to\pm\infty$ and 
\begin{equation}
\mathbb{E}_{J}(u,\eta)=\eta_{-}-\eta_{+}.\label{eqn:energy_equals_action_difference}
\end{equation}
The following theorem shows how the assumption that $\alpha$ is WCRO
implies that one can obtain compactness for flow lines of the perturbed Rabinowitz
action functional $\mathcal{A}_{\varphi}^{\kappa}$. The proof uses
the procedure developed by Cieliebak-Mohnke in \cite{CieliebakMohnke2005}
to establish SFT compactness (compare also \cite{BourgeoisEliashbergHoferWysockiZehnder2003}). We emphasise that in contrast to the standard SFT compactness results, in Theorem \ref{thm:compactness}, we do \emph{not} need to assume that the contact form $\alpha$ is non-degenerate.
\begin{thm}
\label{thm:compactness} Let $\varphi$ be a non-degenerate path of
contactomorphisms and choose $\kappa>\kappa(\varphi)$. Fix $\kappa>\kappa'>\kappa(\varphi)$ and let $J$ be a family of almost complex structures on $S\Sigma$ compatible with $d\lambda$
which is independent of $t$ and of $SFT$-type on the complement
of $(e^{-\kappa'},e^{\kappa'})\times\Sigma$. Then there exists $\ell>0$
such that
\begin{equation}
\mbox{\emph{im}}(u)\subset[e^{-\ell},e^{\kappa}]\times\Sigma\ \ \ \mbox{for all }(u,\eta)\in\mathcal{M}(\varphi,\kappa,J).
\end{equation}
\end{thm}
We will prove Theorem \ref{thm:compactness} in Section \ref{sec:SFT-compactness}
below. The upshot of Theorem \ref{thm:compactness} is that it is
possible to define the \emph{Rabinowitz Floer homology }$\mbox{RFH}_{*}(\Sigma,\alpha;\varphi)$ working directly in the symplectization.
This is a semi-infinite dimensional Morse theory associated to the
functional $\mathcal{A}_{\varphi}^{\kappa}$ (for some $\kappa>\kappa(\varphi)$).
We refer to \cite{CieliebakFrauenfelder2009}, \cite{AlbersFrauenfelder2010c}, and \cite{AlbersMerry2013a} for more details of the construction
used in this paper. Instead here we only summarize the key properties that
we will need about the Rabinowitz Floer homology $\mbox{RFH}_{*}(\Sigma,\alpha;\varphi)$. From now on we will assume that $\alpha$ is WCRO.

\begin{enumerate}
\item \label{enu:The-Rabinowitz-Floer1}Since $\alpha$ is assumed to be WCRO, the Rabinowitz Floer homology $\mbox{RFH}_{*}(\Sigma,\alpha):=\mbox{RFH}_{*}(\Sigma,\alpha;\mbox{id})$
is canonically isomorphic to the singular homology $\mbox{H}_{*+n-1}(\Sigma;\mathbb{Z}_{2})$.
Moreover the Rabinowitz Floer homology is independent of $\varphi$
in the following strong sense: there are canonical isomorphisms 
\begin{equation}
\zeta_{\varphi}:\mbox{RFH}_{*}(\Sigma,\alpha)\rightarrow\mbox{RFH}_{*}(\Sigma,\alpha;\varphi),
\end{equation}
and given two paths $\varphi$ and $\psi$, there is a canonical map $\zeta_{\varphi,\psi}:\mbox{RFH}_{*}(\Sigma,\alpha;\varphi)\rightarrow\mbox{RFH}_{*}(\Sigma,\alpha;\psi)$, the continuation homomorphism,
with the property that 
\begin{equation}
\zeta_{\psi}=\zeta_{\varphi,\psi}\circ\zeta_{\varphi}.\label{eq:continuation homomorphisms}
\end{equation}
In particular, $\mbox{RFH}_{n}(\Sigma,\alpha;\varphi)$ contains a non-zero class $[\Sigma_\varphi]$ which is defined by
\begin{equation}
\zeta_{\varphi,\psi}\big([\Sigma_\varphi]\big)=[\Sigma_\psi]\quad\text{and}\quad [\Sigma_\id]=[\Sigma]\in\mbox{RFH}_{n}(\Sigma,\alpha)=\mbox{H}_{2n-1}(\Sigma;\mathbb{Z}_{2}).
\end{equation}

\item 
\label{properties of RFH}
Denote by $\mbox{RFH}_{*}^{c}(\Sigma,\alpha;\varphi)$ the Rabinowitz
Floer homology generated by the subcomplex of critical points $(u,\eta)$
of $\mathcal{A}_{\varphi}^{\kappa}$ with $\eta\leq c$. Then there
is natural map 
\[
\iota_{\varphi}^{c}:\mbox{RFH}_{*}^{c}(\Sigma,\alpha;\varphi)\rightarrow\mbox{RFH}_{*}(\Sigma,\alpha;\varphi)
\]
induced by the inclusion of critical points. Moreover, given two paths $\varphi$ and $\psi$, there is a constant $K(\varphi,\psi)$ %with the property that $K(\varphi,\psi) \rightarrow 0$ as $\mathrm{dist}_{C^2}(\varphi,\psi) \rightarrow 0$, 
such that the map $\zeta_{\varphi,\psi}$ from Property (1) defines a map
\begin{equation}
\label{where_zeta_goes}
\zeta_{\varphi,\psi}:\mbox{RFH}_{*}^{c}(\Sigma,\alpha;\varphi)\rightarrow\mbox{RFH}_{*}^{c+K(\varphi,\psi)}(\Sigma,\alpha;\psi)
\end{equation}
for any $c \in \mathbb{R}$. We can estimate
\beq\label{eq:the constant K}
K(\varphi,\psi)\leq e^{\max\{\kappa(\varphi),\kappa(\psi)\}}\int_{0}^{1}\max\Big\{\max_{x\in\Sigma}\Big(h_{t}(x)-k_{t}(x)\Big),0\Big\}dt
%\int_{0}^{1}\max_{x\in\Sigma}\big(h_{t}(x)-k_t(x)\big)dt
\eeq
where $h$ and $k$ are the contact Hamiltonians for $\varphi$ and $\psi$, respectively, and $\kappa(\varphi)$ is defined in \eqref{eq:def_kappa_varphi}.

%\item The Rabinowitz Floer homology groups $\mbox{RFH}_{*}(\Sigma,\alpha;\varphi)$
%carry additional algebraic structures \cite{AbbondandoloMerry2013, AbbondandoloMerryOancea2013}.
%That is, given $\varphi$ and $\psi$ one can define a product 
%\begin{equation}
%\mbox{RFH}_{i}^{c}(\Sigma,\alpha;\varphi)\otimes\mbox{RFH}_{j}^{c'}(\Sigma,\alpha;\psi)\rightarrow\mbox{RFH}_{i+j-n}^{c+c'}(\Sigma,\alpha;\varphi\psi).
%\end{equation}
%Taking $\varphi=\psi=\mbox{id}$ the product defines the structure
%of a unital ring on $\mbox{RFH}_{*}(\Sigma,\alpha)$. Under the isomorphism
%$\mbox{RFH}_{*}(\Sigma,\alpha)$ from (1), the product agrees with
%the intersection product on $\mbox{H}_{*+n-1}(\Sigma;\mathbb{Z}_{2})$.
%In particular, the unit $\mathbf{1}_{\textrm{id}}\in\mbox{RFH}_{n}(\Sigma,\alpha)$ corresponds precisely to the fundamental class $[\Sigma]\in\mbox{H}_{2n-1}(\Sigma;\mathbb{Z}_{2})$.
%We define the ``unit'' $\boldsymbol{1}_{\varphi}\in\mbox{RFH}_{n}(\Sigma,\alpha;\varphi)$
%by
%\begin{equation}
%\mathbf{1}_{\varphi}=\zeta_{\varphi}(\mathbf{1}),
%\end{equation}
%(note though that for $\varphi\ne\mbox{id}$, $\mathbf{1}_{\varphi}$
%is not strictly a unit in any sense of the word!). 
\end{enumerate}
\begin{rem}
\label{rem: pair of pants}
%In Property (3) we are using the fact that
%the analogue of Theorem \ref{thm:compactness} holds for maps $(u,\eta)$
%defined on a pair of pants or a disk as well. This is needed to obtain
%the necessary compactness results for the product and the unit alluded
%to above. In fact, in Theorem \ref{thm: the actual result} we will prove a more general result, which we then use to deduce both Theorem \ref{thm:compactness} and its analogue for such maps; see the last paragraph of the proof of Theorem \ref{thm:compactness} \vpageref{adaptions for rem on pairs of pants}. Similarly 
In Property (2) we are using the fact that Theorem \ref{thm:compactness} also holds for $s$-dependent spaces $\mathcal{M}\left( \{ \varphi_s \}_{ s \in [0,1]}, \kappa, \{ J_s\}_{s \in [0,1]}\right)$. In fact, in Theorem \ref{thm: the actual result} we will prove a more general result, which we then use to deduce both Theorem \ref{thm:compactness} and its analogue for $s$-dependent solutions.
\end{rem}

Even though it is more or less standard, the estimate \eqref{eq:the constant K} is extremely important in all that follows, and hence we prove it here. 
To define the continuation homomorphism $\zeta_{\varphi,\psi}$ we denote by $H_t=rh_t$ and $K_t=rk_t$ the Hamiltonian functions of $\varphi$ and $\psi$, respectively, and choose a linear homotopy
\begin{equation}
L^s_t:=\beta(s)H_t +(1-\beta(s))K_t
\end{equation}
for a smooth function $\beta:\R\to[0,1]$ with $\beta(s)=1$ for $s\leq-1$, $\beta(s)=0$ for $s\geq1$ and $\beta'(s)\leq0$. We define the ($s$-dependent) action functional $ \mathcal{A}_s = \mathcal{A}_{ \varphi_s}^{ \kappa, \kappa'}$ as in \eqref{eq:Rab functional}:
\begin{equation}
\mathcal{A}_{s}(u,\eta)=\int_{0}^{1}u^{*}\lambda-\eta\int_{0}^{1}\nu(t)m_{\kappa,\kappa'}(a(t))dt-\int_{0}^{1}\dot{\chi}(t)\varepsilon_{\kappa,\kappa'}(a(t))L^s_{\chi(t)}(u(t))dt,
\end{equation}
where $ \varphi_s$ has corresponding Hamiltonian function $L^s_t$. 
Then counting solutions of 
\begin{equation}
\partial_{s}(u,\eta)+\nabla_{J}\mathcal{A}_{s}(u,\eta)=0
\end{equation}
with $(u_-,\eta_-):=\big(u(-\infty),\eta(-\infty)\big)\in\mbox{Crit}(\mathcal{A}_{\varphi}^{ \kappa})$ and $(u_+,\eta_+):=\big(u(+\infty),\eta(+\infty)\big)\in\mbox{Crit}(\mathcal{A}_{\psi}^{ \kappa})$
defines the continuation homomorphism. We recall that $\kappa>\max\{\kappa(\varphi),\kappa(\psi)\}$ and estimate
\bea
0&\leq\mathbb{E}_{J}(u,\eta)\\
&=\int_{-\infty}^{\infty}\int_{0}^{1}|\partial_{s}(u,\eta)|_{J}^{2}dtds\\
%&=\int_{-\infty}^{\infty}\int_{0}^{1}\left\langle \left\langle \partial_{s}(u,\eta),\partial_{s}(u,\eta)\right\rangle \right\rangle _{J}dtds\\
&=-\int_{-\infty}^{\infty}\int_{0}^{1}\left\langle \left\langle \nabla \mathcal{A}_{s}(u,\eta),\partial_{s}(u,\eta)\right\rangle \right\rangle _{J}dtds\\
&=-\int_{-\infty}^{\infty}\int_{0}^{1}\frac{d}{ds} \mathcal{A}_s(u,\eta)dtds+\int_{-\infty}^{\infty}\int_{0}^{1}\frac{\partial\mathcal{A}_s}{\partial s}(u,\eta)dtds\\
&=\mathcal{A}_{\varphi}(u_-,\eta_-)-\mathcal{A}_{\psi}(u_+,\eta_+)-\int_{-\infty}^{\infty}\int_{0}^{1}\dot{\chi}(t)\varepsilon_{\kappa,\kappa'}(a(t))\frac{\partial L^s_{\chi(t)}}{\partial s}(u(t))dtds\\
&=\mathcal{A}_{\varphi}(u_-,\eta_-)-\mathcal{A}_{\psi}(u_+,\eta_+)\\
&\qquad\qquad-\int_{-\infty}^{\infty}\int_{0}^{1}\beta'(s)\varepsilon_{\kappa,\kappa'}(a(t))\dot{\chi}(t)\Big(H_{\chi(t)}(u(t))-K_{\chi(t)}(u(t))\Big)dtds\\
&=\mathcal{A}_{\varphi}(u_-,\eta_-)-\mathcal{A}_{\psi}(u_+,\eta_+)\\
&\qquad\qquad-\int_{-\infty}^{\infty}\int_{0}^{1}\underbrace{\beta'(s)}_{\leq 0}\underbrace{\varepsilon_{\kappa,\kappa'}(a(t))a(t)}_{0\leq\cdot\leq e^\kappa}\underbrace{\dot{\chi}(t)}_{\geq0}\Big(h_{\chi(t)}(u(t))-k_{\chi(t)}(u(t))\Big)dtds\\
&\leq\mathcal{A}_{\varphi}(u_-,\eta_-)-\mathcal{A}_{\psi}(u_+,\eta_+)\\
&\qquad\qquad-\int_{-\infty}^{\infty}\int_{0}^{1}\underbrace{\beta'(s)}_{\leq 0}\underbrace{\varepsilon_{\kappa,\kappa'}(a(t))a(t)}_{0\leq\cdot\leq e^\kappa}\underbrace{\dot{\chi}(t)}_{\geq0}\max_{x\in\Sigma}\Big(h_{\chi(t)}(x)-k_{\chi(t)}(x)\Big)dtds\\
&\leq\mathcal{A}_{\varphi}(u_-,\eta_-)-\mathcal{A}_{\psi}(u_+,\eta_+)\\
&\qquad\qquad-\int_{-\infty}^{\infty}\int_{0}^{1}\underbrace{\beta'(s)}_{\leq 0}\underbrace{\varepsilon_{\kappa,\kappa'}(a(t))a(t)}_{0\leq\cdot\leq e^\kappa}\underbrace{\dot{\chi}(t)}_{\geq0}\max\Big\{\max_{x\in\Sigma}\Big(h_{\chi(t)}(x)-k_{\chi(t)}(x)\Big),0\Big\}dtds\\
&\leq\mathcal{A}_{\varphi}(u_-,\eta_-)-\mathcal{A}_{\psi}(u_+,\eta_+)\\
&\qquad\qquad-e^\kappa\int_{-\infty}^{\infty}\int_{0}^{1}\beta'(s)\dot{\chi}(t)\max\Big\{\max_{x\in\Sigma}\Big(h_{\chi(t)}(x)-k_{\chi(t)}(x)\Big),0\Big\}dtds\\
&=\mathcal{A}_{\varphi}(u_-,\eta_-)-\mathcal{A}_{\psi}(u_+,\eta_+)\\
&\qquad\qquad-e^\kappa\underbrace{\int_{-\infty}^{\infty}\beta'(s)ds}_{=-1}\int_{0}^{1}\dot{\chi}(t)\max\Big\{\max_{x\in\Sigma}\Big(h_{\chi(t)}(x)-k_{\chi(t)}(x)\Big),0\Big\}dt\\
&=\mathcal{A}_{\varphi}(u_-,\eta_-)-\mathcal{A}_{\psi}(u_+,\eta_+)+e^\kappa\int_{0}^{1}\max\Big\{\max_{x\in\Sigma}\Big(h_t(x)-k_t(x) \Big),0\Big\}dt\\
\eea

This proves estimate \eqref{eq:the constant K}. We conclude this section by proving Theorem \ref{thm:margherita}
from the Introduction, which we restate here for the convenience of
the reader. 
\begin{thm}
Let $(\Sigma,\xi=\ker\,\alpha)$ be a contact
manifold such that $\alpha$ has no contractible
closed Reeb orbits. Then every non-degenerate $\psi\in\mbox{\emph{Cont}}_{0}(\Sigma,\xi)$
has at least $\sum_{j=0}^{\dim\,\Sigma}\dim\,\mbox{\emph{H}}_{j}(\Sigma;\mathbb{Z}_{2})$ many translated points.
\end{thm}
\begin{proof}
Since $\mbox{RFH}_{*}(\Sigma,\alpha;\varphi)\cong\mbox{H}_{*+n-1}(\Sigma;\mathbb{Z}_{2})$,
and the translated points of $\varphi_{1}$ are the generators of
$\mbox{RFH}_{*}(\Sigma,\alpha;\varphi)$, the theorem is almost immediate.
There is a slight subtlety however, coming from the fact that $\varphi_1$
could have translated points lying on a closed Reeb orbit. Indeed,
suppose $x$ is a translated point of $\varphi_1$ lying on a closed
Reeb orbit of period $T$ (note by assumption this orbit is necessarily
non-contractible). Let $0\leq\eta<T$ denote a time-shift of $x$.
Let $y_{k}:\mathbb{R}/2\mathbb{Z}\rightarrow\Sigma$ denote the continuous and piecewise smooth map defined by 
\begin{equation}
y_{k}(t):=\begin{cases}
\theta^{(kT+\eta)t}(\theta^{-\eta}(x)), & t\in[0,1]\\
\varphi_{t-1}(x), & t\in[1,2].
\end{cases}
\end{equation}
Next we observe that there is at most
one value of $k$ for which $y_{k}$ is contractible, and for this value of $k$ we obtain a
generator $(u_{k},\eta+kT)$ of $\mbox{RFH}_{*}(\Sigma,\alpha;\varphi)$. This proves the theorem. 
\end{proof}

\section{Spectral invariants}

We now explain
how to use the Rabinowitz Floer homology groups $\mbox{RFH}_{*}(\Sigma,\alpha;\varphi)$
to define \emph{spectral numbers }$c(\varphi)\in\mathbb{R}$ associated to
any path $\varphi=\{\varphi_{t}\}_{0\leq t\leq1}\in\mathcal{P}\mbox{Cont}_{0}(\Sigma,\xi)$. In particular, we still assume that $\alpha$ is WCRO.
\begin{defn}
\label{def of c}Let $\varphi$ denote a non-degenerate path. We define
its \emph{spectral number} by
\begin{equation}
c(\varphi):=\inf\left\{ c\in\mathbb{R}\mid[ \Sigma_{\varphi}] \in\iota_{\varphi}^{c}(\mbox{RFH}_{*}^{c}(\Sigma,\alpha;\varphi))\right\} .\label{eq:def of c}
\end{equation}

\end{defn}

\begin{prop}\label{prop:the comparison prop}
Let $\varphi$ and $\psi$ be two non-degenerate paths. Then we have the estimate
\bea\label{eqn:estimate_1}
c(\psi)&\leq c(\varphi)+K(\varphi,\psi) \\
%&\leq c(\varphi)+ e^{\max\{\kappa(\varphi),\kappa(\psi)\}}\int_{0}^{1}\max_{x\in\Sigma}\big(h_{t}(x)-k_t(x)\big)dt,\\
&\leq c(\varphi)+ e^{\max\{\kappa(\varphi),\kappa(\psi)\}}\int_{0}^{1}\max\Big\{\max_{x\in\Sigma}\Big(h_{t}(x)-k_{t}(x)\Big),0\Big\}dt
\eea
where $h$ and $k$ are the contact Hamiltonians of $\varphi$ and $\psi$, respectively. In particular, we have
\beq\label{eqn:estimate_2}
h_t(x)\leq k_t(x)\; \forall x\in\Sigma,t\in[0,1]\quad\Longrightarrow\quad c(\varphi)\geq c(\psi).
\eeq
\end{prop}

\begin{proof}
This follows immediately from the definition of the spectral number together with \eqref{where_zeta_goes} and the estimate \eqref{eq:the constant K}. 
\end{proof}

\begin{lem}\label{lem: properties of c}
For any non-degenerate path $\varphi\in\mathcal{P}\mbox{\emph{Cont}}_{0}(\Sigma,\xi)$ its spectral number is a critical value of $\mathcal{A}_\varphi$, i.e.~$c(\varphi)\in\mbox{\emph{Spec}}(\varphi)$. In particular
$c(t\mapsto\theta^{tT})=-T$ where as always $\theta^t$ is the Reeb flow.

Moreover $c$ admits a unique extension
to all of $\mathcal{P}\mbox{\emph{Cont}}_{0}(\Sigma,\xi)$: given
a degenerate path $\varphi$, set 
\begin{equation}
c(\varphi):=\lim_{k}c(\varphi_{k}),\label{eq:how defined in degen case-1}
\end{equation}
where $\varphi_{k}\rightarrow\varphi$ is any sequence of non-degenerate
paths converging to $\varphi$ in $C^2$. The extension still satisfies $c(\varphi)\in\mbox{\emph{Spec}}(\varphi)$ and the estimates \eqref{eqn:estimate_1} and \eqref{eqn:estimate_2}. In particular, $c:\mathcal{P}\mathrm{Cont}_0(\Sigma,\xi) \rightarrow \mathbb{R}$ is a continuous function when we equip $\mathcal{P}\mathrm{Cont}_0(\Sigma,\xi)$ with the $C^2$-topology. 
\end{lem}

\begin{proof}
The assertion $c(\varphi)\in\mbox{Spec}(\varphi)$ follows immediately from the fact that $\mbox{RFH}_{*}^{c}(\Sigma,\alpha;\varphi)$ only changes if $c$ crosses a critical value of $\mathcal{A}_\varphi$. Moreover $\mathrm{Spec}(t\mapsto\theta^{tT})=\{-T\}$.

To prove the existence of the extension we are required to prove that the limit exists and is independent of the choice of approximating sequence $\varphi_k$. We denote by $h_k$ the corresponding contact Hamiltonians. Since we assume that $\varphi_k$ converges to $\varphi$ in $C^2$ it follows that $\kappa(\varphi_k)\to\kappa(\varphi)$ and $h_k\to h$, the contact Hamiltonian of $\varphi$. From Proposition \ref{prop:the comparison prop} we conclude that $(c(\varphi_k))$ converges and in the same way independence of the approximating sequence $(\varphi_k)$ is proved. That  $c(\varphi)\in\mbox{Spec}(\varphi)$ and that the estimates \eqref{eqn:estimate_1} and \eqref{eqn:estimate_2} hold follows from the definition of $c$ as a limit.
\end{proof}

\begin{cor}
\label{cor:posi and negi}
Suppose that $\varphi$ has contact Hamiltonian
$h_{t}$ with $h_{t}\leq -\delta<0$ for all $t\in[0,1]$. Then $c(\varphi)>0$. Similarly if $h_{t}\geq\delta>0$
then $c(\varphi)<0$.
\end{cor}

\begin{proof}
Note that the constant function $\delta$ generates the path $\{t\mapsto\theta^{t\delta}\}$. Thus Proposition \ref{prop:the comparison prop} and Lemma \ref{lem: properties of c} together with $h_{t}\leq -\delta<0$ imply
\begin{equation}
c(\varphi)\geq c(t\mapsto\theta^{-t\delta})=\delta>0.
\end{equation}
Similarly, $h_{t}\geq \delta>0$ implies
\begin{equation}
c(\varphi)\leq c(t\mapsto\theta^{t\delta})=-\delta<0.
\end{equation}
\end{proof}

\begin{lem}
\label{lem:descends}
The map $c:\mathcal{P}\mathrm{Cont}_0(\Sigma,\xi) \rightarrow \mathbb{R}$ descends to give a well defined map $c:\widetilde{\mathrm{Cont}}_0(\Sigma,\xi) \rightarrow \mathbb{R}$.
\end{lem}

\begin{proof}
We recall from Remark \ref{rem: the functional is Morse Bott} that $\mbox{Spec}(\varphi)\subset \R$ is nowhere dense and actually only depends on the endpoint $\varphi_1$ of the path $\varphi$. Moreover, Lemma \ref{lem: properties of c} implies that $c$ is a continuous map. If we vary the path $\varphi$ while keeping the endpoints fixed the continuous map $c$ takes values in the fixed, nowhere dense set $\mbox{Spec}(\varphi_1)$, thus is constant. This proves the Lemma.
\end{proof}

The following result pertains specifically to loops $\varphi=\{\varphi_{t}\}_{t\in S^{1}}$
of contactomorphisms. Recall from Remark \ref{rem:u phi} that to such a loop we have associated classes $u_{\varphi}\in\pi_{1}(\Sigma)$, as well as a class $\tilde{u}_{\varphi}\in[S^{1},\Sigma]$.
We remind the reader of the convention adopted in Remark \ref{rem:inverse para}.
\begin{prop}
\label{prop:loops}Suppose $\varphi$ is a loop of contactomorphisms. Then $c(\varphi)=0$ if and only if $\tilde{u}_{\varphi}$ is the
class of contractible loops. Moreover if $c(\varphi)\ne0$ then there
exists a Reeb orbit of period $-c(\varphi)$ belonging to the free
homotopy class $\tilde{u}_{\varphi}$.\end{prop}
\begin{proof}
We first remind the reader that since $\varphi$ is assumed to be
a loop, if $(u=(a,f),\eta)$ is a critical point of $\mathcal{A}_{\varphi}^{\kappa}$
then the loop $f:S^{1}\rightarrow\Sigma$ is obtained by concatenating
a closed Reeb orbit $t\mapsto\theta^{t\eta}(x)$ of period $\eta$
with the loop $x\mapsto\varphi_{t}(x)$ (modulo reparametrization).\\
Suppose that $c(\varphi)=0$. Then there is a critical point $(u=(a,f),\eta = 0)$
of $\mathcal{A}_{\varphi}^{\kappa}$ which (modulo reparametrization)
is of the form $f(t)=\varphi_{t}(x)$ for some point $x\in\Sigma$.
Since $\mathcal{A}^{\kappa}_{\varphi}$ is defined on the space of contractible loops, $u$ is contractible, and thus the class $\tilde{u}_{\varphi}$ is necessarily
trivial. If $c(\varphi)\ne0$ then there exists a critical point of
$\mathcal{A}_{\varphi}^{\kappa}$ of the form $(u=(a,f),\eta)$ where
$f$ is the concatenation of a (non-constant) closed Reeb orbit and
the loop $x\mapsto\varphi_{t}(x)$ for some point $x\in\Sigma$. Since
$u$ is a contractible loop by assumption, this Reeb orbit must belong to the free
homotopy class $-\tilde{u}_{\varphi}$. This proves the remaining two statements. 
\end{proof}
We can now prove the main results of this paper.
\begin{proof}[Proof of Theorem \ref{thm:main} and of the sharpened statement from Remark \ref{rem:u phi}]
\label{proof of theorem main}
If $(\Sigma,\xi)$ is not hypertight then there is nothing to prove.
Otherwise there exists $\alpha$ which is WCRO. We first show that there are no non-trivial contractible positive loops. Indeed, if $\varphi$ is a contractible loop then $c(\varphi)=c(\mathrm{id})=0$ by Lemma \ref{lem:descends}. But from Corollary \ref{cor:posi and negi}, a positive loop $\varphi$ satisfies $c(\varphi)<0$. Now assume that $\mathrm{Cont}_0(\Sigma,\xi)$ is non-orderable. Thus there exists
a positive loop $\varphi$ (which is necessarily non-contractible). Then again by Corollary \ref{cor:posi and negi} we have $c(\varphi)<0$, and Proposition \ref{prop:loops} implies that 
$\tilde{u}_{\varphi}$ contains a Reeb orbit of period $-c(\varphi)$. The final statement from Remark \ref{rem:u phi} follows again from Proposition \ref{prop:loops}.
\end{proof}

\begin{proof}[Proof of Theorem \ref{thm:main2}]
Choose a supporting contact form $\alpha$ for $(\Sigma, \xi)$ which is WCRO. We argue by contradiction: assume that
there is a positive loop $\varphi$ and $x\in\Sigma$ such that the class $u_{\varphi}\in\pi_{1}(\Sigma ,x )$
is torsion. Thus there is $k\in\mathbb{N}$ with 
\begin{equation}
1=(u_{\varphi})^{k}=u_{\varphi^{k}} \in \pi_1({\Sigma}).
\end{equation}
This implies that $\tilde{u}_{\varphi^{k}} \in [S^1, \Sigma]$ is the class of contractible
loops. Hence by Proposition \ref{prop:loops} one has $c(\varphi^{k})=0$.
But $\varphi^{k}$ is still a positive loop, and hence $c(\varphi^{k})<0$ by Corollary \ref{cor:posi and negi}; contradiction.
\end{proof}

\section{\label{sec:SFT-compactness}SFT compactness}

The aim of this section is to prove Theorem \ref{thm:compactness}.
In fact, we will prove in Theorem
\ref{thm: the actual result} below a result on pseudoholomorphic curves 
which will imply Theorem \ref{thm:compactness}, and the generalization mentioned in Remark \ref{rem: pair of pants}. Before stating
Theorem \ref{thm: the actual result}, we need to introduce various
definitions.  

In this section, it will be more convenient to view the symplectization
of $\Sigma$ as the manifold $\mathbb{R}\times\Sigma$ endowed with
the symplectic form $d(e^{s}\alpha)$, where $s$ is the coordinate
on $\mathbb{R}$. This identification is justified by the canonical
map 
\begin{equation}
i:\mathbb{R}\times\Sigma\longrightarrow(0,\infty) \times \Sigma \quad(s,x)\mapsto(r,x):=(e^s,x)\label{eq:the identification i}
\end{equation}
which satisfies $i^{*}(r\alpha)=e^{s}\alpha$ and is therefore an
exact symplectomorphism. Under this identification an almost complex
structure on $(0,\infty)\times\Sigma$ of SFT-type is identified with
an almost complex structure $i^{*}J$ on $\mathbb{R}\times\Sigma$
which is invariant under $\mathbb{R}$-translations, preserves the
contact distribution and satisfies $JR=\partial_{s}$, where $R=(0,R)$
still denotes the Reeb vector field for $\alpha$. The set of such
almost complex structures will be denoted by $\mathcal{J}_{\textrm{SFT}}(\mathbb{R}\times\Sigma,d(e^{s}\alpha))$.
We next recall the definition of the \emph{Hofer energy }of a $J$-holomorphic
map $u:Z\rightarrow\mathbb{R}\times\Sigma$. 
\begin{defn}
\label{defn:HoferEnergy} Let $(Z,j)$ denote a compact Riemann surface
(possibly disconnected and with boundary). Orient $Z$ by declaring
$(jv,v)$ to be a positively oriented basis of $T_zZ$ whenever $0\neq v\in T_zZ$.
 Fix $J\in\mathcal{J}_{\textrm{SFT}}(\mathbb{R}\times\Sigma,d(e^{s}\alpha))$.
Suppose $u:Z\rightarrow\mathbb{R}\times\Sigma$ is a $(j,J)$-holomorphic
map. Write $u=(a,f)$ so that $f:Z\rightarrow\Sigma$ and $a:Z\rightarrow\mathbb{R}$.
Define the \emph{Hofer energy }$E(u)$ of $u$ as 
\begin{equation}
E(u)=\sup_{\nu\in\mathcal{S}}\int_{Z}u^{*}d(\nu\alpha)=\sup_{\nu\in\mathcal{S}}\left(\int_{Z}u^{*}(\nu d\alpha)+\int_{Z}u^{*}(\nu'(s)ds\wedge\alpha)\right)\in [0,\infty],
\end{equation}
where $\mathcal{S}:=\left\{ \nu\in C^{\infty}(\mathbb{R},[0,1])\mid\nu'\geq0\right\} $.
\end{defn}
Given $J\in\mathcal{J}_{\textrm{SFT}}(\mathbb{R}\times\Sigma,d(e^{s}\alpha))$,
we can define a compatible $\mathbb{R}$-invariant Riemannian metric
$g_{J}$ on $\mathbb{R}\times\Sigma$ given by $g_{J}(\cdot,\cdot)=d\alpha(J\cdot,\cdot)+\alpha^{2}(\cdot,\cdot)+(dr)^{2}(\cdot,\cdot)$.\\
\emph{Warning: }Since the compatibility condition imposed on $J$ 
follows the (slightly unusual) sign convention adopted throughout this paper, the first factor of the metric $g_J$ takes a different form to the corresponding metric in \cite{CieliebakMohnke2005}.
\begin{example}
Assume $\gamma$ is a $T$-periodic orbit of $\alpha$ and let
$u:\R\times S^1\rightarrow \R\times \Sigma$ have the form $u(s,t)=(c\pm Ts,
\gamma(\pm tT))$ for some $c\in \R$. If the complex structure $j$ on $\R\times S^1$ satisfies $j\partial_t=\partial_s$ for coordinates $(s,t)\in \R\times S^1$, then $u$ is a $J$-holomorphic map for any $J\in \mathcal{J}_{\textrm{SFT}}(\mathbb{R}\times\Sigma,d(e^{s}\alpha))$, and $E(u)=T$.
Such a map $u$ is called a \emph{trivial cylinder over the periodic orbit $\gamma$}.
\end{example}
%\begin{rem}
%\label{rem: a bit more general}Although not necessary in this paper,
%rather than assuming $\alpha$ is WCRO, in Theorem \ref{thm: the actual result}
%below we assume instead that all contractible Reeb orbits of $\alpha$
%have period $>E$ for some constant $E$ (thus
%in this paper we may as well take $E=\infty$). We will use this generalization
%elsewhere.
%\end{rem}
Next we state the main result of this section. It is probably well known to those who studied work related to SFT compactness as in \cite{BourgeoisEliashbergHoferWysockiZehnder2003},\cite{CieliebakMohnke2005}. Nevertheless, to the best of our knowledge, this result has not appeared explicitly in the literature so far. We emphasise that the novelty in the statement is that we make \emph{no} non-degeneracy assumptions on the Reeb orbits of $\alpha$. We give below a proof following the method of \cite{CieliebakMohnke2005}, but it is also possible to prove this result using Fish's \emph{target local compactness} \cite{Fish2011}, as we explain in Remark \ref{rem:TLC} below. 
\begin{thm}
\label{thm: the actual result} Let $(\Sigma,\alpha)$ be a cooriented contact manifold and let $J\in\mathcal{J}_{\textrm{\emph{SFT}}}(\mathbb{R}\times\Sigma,d(e^{s}\alpha))$.  
Suppose $(Z_{k},j_{k})$ is a family of compact (possibly disconnected) Riemann surfaces with boundary and uniformly bounded genus. Assume that 
\begin{equation}
u_{k}=(a_{k},f_{k}):Z_{k}\rightarrow \R\times\Sigma
\end{equation}
is a sequence of $(j_{k},J)$-holomorphic maps which have uniformly bounded Hofer energy $E(u_{k})\leq E$, are nonconstant on each connected
component of $Z_k$, and satisfy $a_{k}(\partial Z_{k})\subset [0,\infty)$. \\
Assume that $\inf_k \inf_{Z_k} a_k=-\infty$. Then there exists a subsequence $k_n$
and cylinders $C_n\subset Z_{k_n}$, biholomorphically equivalent to
standard cylinders $[-L_n,L_n]\times S^1$, such that
$L_n\rightarrow \infty$ and such that $u_{k_n}\vert_{C_n}$  converges (up to an $\R$-shift) in $C^\infty_{\emph{loc}}(\R\times S^1,\R\times \Sigma)$
to a trivial cylinder over a Reeb orbit of period at most $E$.
\end{thm}

We now show how Theorem \ref{thm:compactness} follows from Theorem \ref{thm: the actual result}. 
\begin{proof}[Proof of Theorem \ref{thm:compactness}]
Observe that in the
situation of Theorem \ref{thm:compactness},  $J$ is independent of $t$ on the complement of $(e^{-\kappa'},e^{\kappa'})\times\Sigma$ and
the restriction $u\vert_{u^{-1}((0,\infty)\setminus(e^{-\kappa},e^{\kappa}))\times\Sigma)}$
of any flow line $(u,\eta)\in\mathcal{M}(\varphi,\kappa,J)$ is a
$J$-holomorphic curve by equation \eqref{eq:gradientA}. 
Since the asymptotic ends $u_\pm$ of $u$ are contained in $(e^{-\kappa'},e^{\kappa'})\times \Sigma$, we can apply the maximum principle to the $J$-holomorphic curve $u\vert_{u^{-1}((0,\infty)\setminus(e^{-\kappa},e^{\kappa}))\times\Sigma)}$ to see that
$\mbox{im}\, u\subset (0,e^{\kappa}]\times \Sigma$.\\
Recall the map $i$ from \eqref{eq:the identification i} and observe that there exists $\overline{J}\in\mathcal{J}_{\textrm{SFT}}(\mathbb{R}\times\Sigma,d(e^{s}\alpha))$
such that if $(u,\eta)\in\mathcal{M}(\varphi,\kappa,J)$ then the restriction $u\vert_{u^{-1}((-\infty,-\kappa]\times\Sigma)}$ gives rise
to a $\overline{J}$-holomorphic map $\overline{u}$ into $\mathbb{R}\times\Sigma$.\\
\emph{Claim:} The Hofer energy $E(\overline{u})$ for $(u,\eta)\in\mathcal{M}(\varphi,\kappa,J)$ is uniformly bounded
by $e^{\kappa}(\eta_--\eta_+)$.
Indeed, if $\lim_{s\rightarrow\pm\infty}(u,\eta)=(u_{\pm},\eta_{\pm})$,
we have from \eqref{eqn:energy_equals_action_difference} that 
\begin{equation}
\eta_{-}-\eta_{+}=\int_{-\infty}^{\infty}\vert\nabla\mathcal{A}_{\varphi}^{\kappa}(u,\eta)\vert^{2}_{J}\geq\int_{\overline{u}^{-1}((-\infty,-\kappa]\times\Sigma)}\overline{u}^{*}d(e^{s}\alpha).
\end{equation}
On the other hand, by Stokes' Theorem we have for any $\nu\in\mathcal{S}$ that
\begin{equation}
\int_{\overline{u}^{-1}((-\infty,-\kappa]\times\Sigma)}\overline{u}^{*}d(\nu\alpha)\leq\int_{\overline{u}^{-1}((-\infty,-\kappa]\times\Sigma)}\overline{u}^{*}d\alpha=e^{\kappa}\int_{\overline{u}^{-1}((-\infty,-\kappa]\times\Sigma)}\overline{u}^{*}d(e^{s}\alpha)
\end{equation}
and therefore, by the definition of $E(\overline{u})$ (cf. Definition  \ref{defn:HoferEnergy})
the claimed bound follows.\\
Assume now by contradiction that there is no $\ell$ as in Theorem \ref{thm:compactness}, and hence there exists a sequence of gradient flow lines $(u_{k},\eta_{k}) \in\mathcal{M}(\varphi,\kappa,J)$ such that the corresponding maps $\overline{u}_k = (\overline{a}_{k},\overline{f}_{k})$ satisfy $\lim_k \inf\overline{a}_{k}=-\infty$. Then, for a $K<-\kappa$ which is a regular value
of all $\overline{a}_{k}$'s, we 
consider the $\overline{J}$-holomorphic curves $v_k:=\overline{u}_{k}\vert_{\overline{u}_{k}^{-1}((-\infty,K]\times \Sigma)}$
and let $Z_{k}:=(\overline{u}_{k})^{-1}((-\infty,K]\times \Sigma)$. \\
Since the gradient flow lines are asymptotic to critical points contained in $(e^{-\kappa'},e^{\kappa'})\times \Sigma$, each $Z_{k}$ is a compact
Riemann surface of genus $0$. Moreover the $\overline{J}$-holomorphic curves $v_k$ have no constant components and their Hofer energy is uniformly
bounded by $e^{\kappa}(\eta_--\eta_+)$.\\
By applying Theorem \ref{thm: the actual result} to the pseudoholomorphic
curves $v_k$ (shifted by $K$ in the $\R$-direction) it follows that there
exists a map $u_{k_0}$ (in fact a whole subsequence of the $u_{k}$ of such maps) with the following property: there is an embedded circle $S$ in the
domain $\mathbb{R}\times S^{1}$  of $u_{k_0}$,
such that the restriction of $f_{k_0}$ to $S$ parametrizes a circle
in $\Sigma$ homotopic to a Reeb orbit $\gamma$ of period less than
$e^{\kappa}(\eta_{-}-\eta_{+})$.
Since the domain of $u_{k_0}$ is $\mathbb{R}\times S^{1}$,
this circle $S$ bounds a disk $D$ in $\mathbb{R}\times S^{1}$,
or it is isotopic to a circle $\{s\}\times S^{1}\subset\mathbb{R}\times S^{1}$.
In either case $\gamma$ is contractible, since $u_{k_0}(S)$ is contractible.
In the latter case this follows since $f_{k_0}(S)$ is 
homotopic to the asymptotic end of $f_{k_0}$, and this is contractible since the asymptotic ends $(u_{k_0})_\pm$ of $u_{k_0}=(a_{k_0},f_{k_0})$ lie in $\mathcal{L}S\Sigma$. 
This contradiction shows
that images of all the maps $u$ (for $(u,\eta)\in\mathcal{M}(\varphi,\kappa,J)$) must be contained in a compact subset
$[e^{-\ell},e^{\kappa}]\times \Sigma$ as claimed.\\
\emph{Adaptations for Remark \ref{rem: pair of pants}\label{adaptions for rem on pairs of pants}:}  We note that in the $s$-dependent case (i.e. for moduli spaces $\mathcal{M}\left( \{ \varphi_s \}_{ s \in [0,1]}, \kappa, \{ J_s\}_{s \in [0,1]}\right)$) the proof goes through verbatim, since the almost complex structures involved are by assumption both $s$ and $t$ independent  on the complement of a compact subset of the symplectization.

%For future purposes we point out that not only Remark \ref{rem: pair of pants} is correct but also that the compactness statement remains true if the domain of the maps $u_{k}$ is $\R^2$ or a pair of pants. Indeed, we have
%a similar conclusion: either the embedded circle bounds a disk or
%it is homotopic to one of the asymptotics. Indeed, a compact surface $C$ of genus $0$ with at most 3 boundary components is decomposed by an embedded circle in the interior into two compact connected surfaces $C_1,C_2$ of genus $0$ with in total at most 5 boundary components; therefore an embedded circle in $C$ is homotopic to
%one of the boundary components of $C$ unless it bounds a disk in $C$. 

\end{proof}

The proof of Theorem \ref{thm: the actual result}
depends on the following proposition, which we will prove in the next section as Proposition \ref{prop:finding_subcylinder2} by following the methods of \cite{CieliebakMohnke2005}.
\begin{prop}
\label{prop:finding_subcylinder} Under the assumptions of Theorem \ref{thm: the actual result} we can find, by passing to a subsequence,
compact subcylinders $C_{k}\subset Z_{k}$ such that an $\mathbb{R}$-shift
$v_{k}=(b_k,g_k)$ of the restriction of $u_{k}$ to $C_{k}$ has the following
properties: 
\begin{enumerate}
\item $C_{k}$ is biholomorphic to $[-L_{k},L_{k}]\times S^{1}$ and $L_{k}\rightarrow\infty$
as $k\rightarrow\infty$ 
\item $\int_{C_k} v_{k}^{*}d\alpha\rightarrow0$ 
\item There is a sequence $\sigma_{k}\rightarrow\infty$ such that $\pm b_{k}(\pm L_{k},t)\geq\sigma_{k}$
for each $t\in S^{1}$. 
\end{enumerate}
\end{prop}

\begin{proof}[Proof of Theorem \ref{thm: the actual result}]
Theorem \ref{thm: the actual result} is an immediate consequence of Proposition \ref{prop:finding_subcylinder} and the
following claim.\\
\emph{Claim}: If $v_{k}=(b_{k},g_{k}):[-L_{k},L_{k}]\times S^{1}\rightarrow\mathbb{R}\times\Sigma$
is a sequence of pseudoholomorphic cylinders as in Proposition \ref{prop:finding_subcylinder}
above, then there exists a subsequence which converges, after possibly
an $\mathbb{R}$-shift, in the $C_{\textrm{loc}}^{\infty}$-topology to
a trivial cylinder over a periodic Reeb orbit.\\ 
\emph{Proof of Claim:}
First observe that away from the boundary the derivatives are uniformly
bounded. That is, there is $D>0$ such that $\left\Vert dv_k(s,t)\right\Vert \leq D$
for all $(s,t)\in[-L_{k}+1,L_{k}-1]\times S^{1}$. Otherwise, by applying
a bubbling off procedure on disks with radius $1/3$ around points
where the gradient blows up, one would obtain a nonconstant pseudoholomorphic
map $v:\mathbb{C}\rightarrow\mathbb{R}\times\Sigma$ with $\int v^{*}d\alpha=0$;
this however contradicts Lemma 28 in \cite{Hofer1993}. Since we have
uniform gradient bounds, and thereby also uniform bounds on the higher
derivatives, we find by the theorem of Arzel\`a-Ascoli that, for a sequence $\tau_{k}\in\mathbb{R}$,
a subsequence of $\tilde{v}_{k}:=(b_{k}-\tau_{k},g_{k})$ converges
in $C_{\textrm{loc}}^{\infty}$ to a pseudoholomorphic cylinder $\tilde{v}:\mathbb{R}\times S^{1}\rightarrow\mathbb{R}\times\Sigma$.
Since $\int\tilde{v}^{*}d\alpha=0$, it follows from \cite[p538]{Hofer1993} that there are only two options: Either $\tilde{v}$
is constant or $\tilde{v}$ is a trivial cylinder over a periodic
Reeb orbit. Therefore it remains to show that $\tilde{v}$ is nonconstant.
To achieve this, we use arguments from \cite{HoferWysockiZehnder2002}.

If $\tilde{v}$ is constant, then in particular, $\tilde{v}_{k}(0,0)\rightarrow p_{0}\in\mathbb{R}\times\Sigma$
and $\int_{S^{1}}\tilde{v}_{k}(0,\cdot)^{*}\alpha\rightarrow0$. As
$\int\tilde{v}_{k}^{*}d\alpha\rightarrow0$, Stokes' Theorem implies
that $\int_{S^{1}}\tilde{v}_{k}(s,\cdot)^{*}\alpha\rightarrow0$ uniformly
in $s\in[-L_{k},L_{k}]$. We will show now that in this situation,
there is $h>0$ such that 
\begin{equation}
\tilde{v}_{k}(s,t)\in B_{1}(p_{0})\text{ for all }(s,t)\in[-L_{k}+h,L_{k}-h].
\label{eq:finding an h}
\end{equation}
Assume not. Then (up to passing to a subsequence of $\tilde{v}_{k}$),
there is a sequence $(s_{k},t_{k})\in[-L_{k}+k,L_{k}-k]\times S^{1}$
such that $\tilde{v}_{k}(s_{k},t_{k})\notin B_{1}(p_0)$; we may assume
however, without loss of generality, that $\tilde{v}_{k}(s_{k},t_{k})\in B_{2}(p_{0})$
for sufficiently large $k$ (since $[-L_{k}+k,L_{k}-k]\times S^{1}$
is connected and $\tilde{v}_{k}(0,0)\rightarrow p_{0}$). Consider
now the sequence $\tilde{v}_{k}':[-k,k]\times S^{1}\rightarrow\mathbb{R}\times\Sigma$
defined by $\tilde{v}_{k}'(s,t):=\tilde{v}_{k}(s_{k}+s,t_{k}+t)$.
Since the derivatives are bounded away from the boundary and $\tilde{v}_{k}'(0,0)\in B_{2}(p_{0})$,
a subsequence of $\tilde{v}'_{k}$ (without any $\mathbb{R}$-shift)
converges to a pseudoholomorphic map $\tilde{v}':\mathbb{R}\times S^{1}\rightarrow\mathbb{R}\times\Sigma$
with $\int\tilde{v}'^{*}d\alpha=0$. Lemma 28 in \cite{Hofer1993}
combined with $\int_{S^{1}}\tilde{v}_{k}(s_{k},\cdot)^{*}\alpha\rightarrow0$
now show that $\tilde{v}'$ is the constant map $\tilde{v}'(s,t)=\tilde{v}'(0,0)=:p_{1}\neq p_{0}$.
Let us assume without loss of generality $s_{k}>0$ and consider the
pseudoholomorphic cylinders $w_{k}:[0,s_{k}]\times S^{1}\rightarrow\mathbb{R}\times\Sigma$
defined by restriction: $w_{k}(s,t):=\tilde{v}_{k}(s,t)$. They have the following
properties: 
\begin{equation}
\int w_{k}^{*}d\alpha\rightarrow0,\,\sup_{s\in[0,s_{k}]}\int w_{k}(s,\cdot)^{*}\alpha\rightarrow0,\, w_{k}(s_{k},t)\rightarrow p_{1},w_{k}(0,t)\rightarrow p_{0},\,\forall t\in S^{1}.
\end{equation}
There are now $(\sigma_{k},\tau_{k})\in[0,s_{k}]\times S^{1}$ such
that if $q_{k}:=w_{k}(\sigma_{k},\tau_{k})$, then $\mbox{dist}(q_{k},p_{0})\geq1/3$ and $\mbox{dist}(q_{k},p_{1}) \leq 1/3$.
This allows us to apply the monotonicity Lemma \ref{lem:monotonicity}
centered at $q_{k}$ which implies that $\mbox{area}(w_{k})\geq \tfrac{C}{9}$.
On the other hand, the area of $w_{k}$ is bounded by its Hofer energy
$E(w_{k})$ which goes to $0$ by the first two conditions above, and
we have a contradiction. In conclusion, there is an $h>0$ as claimed
in equation \eqref{eq:finding an h} above. 

Thus the map $\tilde{v}_{k}=(\tilde{b}_{k}=b_{k}-\tau_{k},g_{k})$
has the following properties: $\tilde{v}_{k}([-L_{k}+h,L_{k}-h]\times S^{1})\subset B_{1}(p_{0})$,
but $\tilde{b}_{k}(L_{k},t)\geq\sigma_{k}-\tau_{k}$ and $\tilde{b}_{k}(-L_{k},t)\leq-\sigma_{k}-\tau_{k}$.
Since (at least) one of the sequences $\pm\sigma_{k}-\tau_{k}$ is
unbounded, at least one of the pseudoholomorphic cylinders $\tilde{v}_{k}^{+}:[L_{k}-h,L_{k}]\times S^{1}\rightarrow\mathbb{R}\times\Sigma$
and $\tilde{v}_{k}^{-}:[-L_{k},-L_{k}+h]\times S^{1}\rightarrow\mathbb{R}\times\Sigma$
will have arbitrary large conformal modulus (see Remark \ref{rem:modulus} below for the definition of conformal modulus) by Lemma \ref{lem:modulusbound}.
However the conformal modulus of each of them is constant equal to
$h$, a contradiction which proves the claim. 
\end{proof}
\begin{rem}[Communicated by Joel Fish]\label{rem:TLC}
Theorem \ref{thm: the actual result} is also a direct consequence of Fish's work on target local compactness \cite[Theorem A]{Fish2011}. 
An outline of his argument is as follows:\\
First we can ``trim'' the curves so that their new domains are given by $Z_k:=a_{k}^{-1}((-\infty, 0])$ and $a_{k}(Z_k)=[A_k, 0]$ 
with $\lim_{k\to\infty} A_k = -\infty$. Since the Hofer energy is bounded, so is the $d\alpha$ energy, and the latter is additive.  
Consequently by the pigeon-hole principle, we may find constants $c_k<0 $ and $\ell_k>0$ such that $\mathcal{I}_k:=[c_k -\ell_k, c_k+\ell_k]\subset [A_k, 0]$, 
and such that 
\begin{equation}
\int_{a_k^{-1}(\mathcal{I}_k)} f_k^* d\alpha \to 0 \quad\text{and}\quad \ell_k\to \infty.
\end{equation} 
Now, after $\mathbb{R}$-shifting the curves by $c_k$ (or $-c_k$), one can
apply target local compactness inductively on the region $[-1, 1]\times \Sigma$, then on the region $[-2, 2]\times\Sigma$ and so forth. By passing
to a diagonal subsequence we construct a sequence which converges on regions mapped into $[-c, c]\times \Sigma$ for any $c$ in some open dense subset of $\R$. 
Since the $d\alpha$ energy tends to zero, the image of the limit curves must lie in an orbit cylinder. 
Since the genus is bounded and the limit is a branched cover of a finite collection of orbit cylinders, we can conclude that the
number of branch points is a priori bounded.
Hence we can find, after shifting and trimming again (and possibly passing to another subsequence), a sequence $c_k'<0$ such that
for the subsequence of curves shifted by $c_k'$ (or $-c_k'$), we have convergence in the region $[-c, c]\times\Sigma$ to an unbranched
multiple cover of a collection of orbit cylinders for arbitrary $c>0$. To complete the proof, one needs to estimate the conformal modulus of the cylinders; 
this can be done by using results in \cite{Fish2007} or by using Lemma \ref{lem:modulusbound} below.\\
It is perhaps worth noting that if one argues using target local compactness, the topological bounds established in Proposition \ref{prop:allbounds} are
a consequence of the compactness result; in our proof we first establish the topological bounds (by following \cite{CieliebakMohnke2005}) and 
then use elementary compactness arguments to find a trivial cylinder.
\end{rem}

\section{\label{sec:Proof_finding_cylinder}Proof of Proposition \ref{prop:finding_subcylinder}}
In this section we will give a proof of Proposition \ref{prop:finding_subcylinder} by following the procedure used by Cieliebak
and Mohnke in \cite{CieliebakMohnke2005} to establish SFT compactness.
We will present below a slight adaptation of the relevant part of their procedure.\\
First we quote the following \emph{monotonicity lemma} from \cite[Lemma 5.1]{CieliebakMohnke2005}.
A detailed proof in the case of a compact target manifold can be found
in \cite[Chapter 2]{Hummel1997}. Since both $J$ and $g_{J}$ are
$\mathbb{R}$-invariant the proof in our situation is easily reduced
to that case. 
\begin{lem}
\label{lem:monotonicity} Fix $J\in\mathcal{J}_{\textrm{\emph{SFT}}}(\mathbb{R}\times\Sigma,d(e^{s}\alpha))$.
Then there exist constants $C,1>10\varepsilon>0$ such that for any
$J$-holomorphic map $u:Z\rightarrow\mathbb{R}\times\Sigma$ defined
on a (possibly disconnected) compact Riemann surface $Z$ and $0<\delta<\varepsilon$
\begin{equation}
\mbox{\emph{area}}_{g_{J}}(u\vert_{B_{g_{J}}(y;\delta)})\geq C\delta^{2}.
\end{equation}
whenever there is a $y\in u(Z)$ with $B_{g_{J}}(y;\delta)\cap u(\partial Z)=\emptyset$ such that $u$ is nonconstant on a component $Z_{0}\subset Z$ whose image contains
$y$. \end{lem}

The next several lemmata allow us to achieve with Proposition \ref{prop:allbounds} some control over the topology of pseudoholomorphic curves in a symplectization.

Let us now fix constants $0<\delta<\varepsilon$ and $C$ as in Lemma \ref{lem:monotonicity} and let $Z$ be a (not necessarily
connected) compact Riemann surface $Z$ and $u=(a,f):Z\rightarrow\mathbb{R}\times\Sigma$
a $J$-holomorphic curve which is nonconstant on all its components
and satisfies $u(\partial Z)\subset(K+4\delta,\infty)\times\Sigma$
for some $K\in\mathbb{R}$. Then similarly to Cieliebak and Mohnke
in \cite[Section 5.3]{CieliebakMohnke2005}, we introduce surfaces
$Z_{R}^{S}(u)$ and $Z_{S}^{R}(u)$ where $R<S<K$ are such that $R,S,R\pm\delta,S\pm\delta$
are regular values of $a$ and $S-R\geq2\delta$.

Namely, first define $\mathcal{C}_{R}$ to be the collection of connected
components of $a^{-1}([R,R+\delta])$ and $a^{-1}([R-\delta,R])$.
Then define subsets $\mathcal{C}_{R}^{\pm}\subset\mathcal{C}_{R}$
as follows: First, we include in $\mathcal{C}_{R}^{+}$ all components
that meet $a^{-1}(R+\delta)$, as well as those in $a^{-1}([R,R+\delta])$
that do not meet $a^{-1}(R)$. Similarly, we include in $\mathcal{C}_{R}^{-}$
all components that meet $a^{-1}(R-\delta)$ as well as those in $a^{-1}([R-\delta,R])$
that do not meet $a^{-1}(R)$. In the next step, we include in $\mathcal{C}_{R}^{+}$
all components in $\mathcal{C}_{R}$ which can be connected to $\mathcal{C}_{R}^{+}$
without passing through $\mathcal{C}_{R}^{-}$, then we include in
$\mathcal{C}_{R}^{-}$ all components in $\mathcal{C}_{R}$ which
can be connected to $\mathcal{C}_{R}^{-}$ without passing through
(the previously extended) $\mathcal{C}_{R}^{+}$. This last step is
repeated as long as the size of the collections $C_{R}^{\pm}$ increases;
this is a finite process since $R$ is a regular value. Since there
are no closed components, we see that $\mathcal{C}_{R}=\mathcal{C}_{R}^{+}\cup\mathcal{C}_{R}^{-}$.
By an abuse of notation we will denote by $\mathcal{C}_{R}^{\pm}$ also
the subsets of $Z$ which are the union of the components in $\mathcal{C}_{R}^{\pm}$.

Now fix $b\in(K+2\delta,K+3\delta)$ such that $b,b\pm\delta$
are regular values of $a$ (a choice depending on $u=(a,f)$). Depending
on $b$ we set for $R<S<K$ as above 
\begin{equation}
Z_{R}^{S}(u):=a^{-1}([R+\delta,S-\delta])\cup\mathcal{C}_{R}^{+}\cup\mathcal{C}_{S}^{-}\text{ \,\ and \,}Z_{S}^{R}(u):=(a^{-1}((-\infty,b-\delta])\cup\mathcal{C}_{b}^{-})\setminus Z_{R}^{S}(u).
\label{eqn:decomposition}
\end{equation}

Observe that since there are no closed components, each connected
component of $Z_{R}^{S}(u)$ and $Z_{S}^{R}(u)$ has a boundary component.
These boundary components can a priori lie either in $a^{-1}(R),a^{-1}(S)$
or in $a^{-1}(b)$ (this last option $a^{-1}(b)$ does not occur in
the setting of \cite{CieliebakMohnke2005}). \\

In the following Lemma we bound the number of components in
$Z_{R}^{S}(u)$ and $Z_{S}^{R}(u)$. 
\begin{lem}
\cite[Lemma 5.5]{CieliebakMohnke2005}\label{lem:componentbound}
For a pseudoholomorphic curve $u=(a,f)$ and regular values $R,S$ as above,
the number of connected components of the surfaces $Z_{R}^{S}(u)$
and $Z_{S}^{R}(u)$ is at most $8E(u)/C\delta^{2}$. \end{lem}
\begin{proof}
Let $Z_{0}$ be a connected component of $Z_{R}^{S}(u)$ (resp. $Z_{S}^{R}(u)$
which has a boundary component away from $a^{-1}(b)$). Then $Z_{0}$
has a boundary component in $a^{-1}(R)$ or $a^{-1}(S)$ and by construction
there exists therefore a point $z_{0}\in Z_{0}$ with $a_{0}:=a(z_{0})\in\{R+\delta/2,S-\delta/2\}$
(resp. $a_{0}:=a(z_{0})\in\{R-\delta/2,S+\delta/2\}$). On the other
hand, if (some or) all boundary components of $Z_{0}$ meet $a^{-1}(b)$
(and therefore $Z_{0}\subset Z_{S}^{R}(u)$), then it meets both $a^{-1}(b)$
and $a^{-1}(b-\delta)$. Therefore there exists a point $z_{0}\in Z_{0}$
with $a_{0}:=a(z_{0})=b-\delta/2$. Fix a $\nu\in\mathcal{S}$ satisfying
$\nu'(s)\geq1$ whenever $s$ is in a $\frac{\delta}{2}$-ball around
$R\pm\delta/2$, $S\pm\delta/2$ or $b-\delta/2$. (The existence
of $\nu$ is clear since $10\varepsilon<1$). We now see that the
ball of radius $\delta/2$ around $y_{0}:=u(z_{0})$ does not intersect
$u(\partial Z_{0})$ and so by the monotonicity Lemma \ref{lem:monotonicity}
\begin{equation}
\frac{C\delta^2}{4}\leq\mbox{area}_{g_{J}}(u\vert_{a^{-1}(a_{0}-\frac{\delta}{2},a_{0}+\frac{\delta}{2})\cap Z_{0}})\leq\int_{Z_{0}}u^{*}(d\alpha)+\int_{Z_{0}}u^{*}(\nu'(s)ds\wedge\alpha)\leq2E(u)
\label{eqn:monotonicity_application}
\end{equation}
Here we used that 
\begin{equation}
\left\vert \frac{du(z)}{dx}\right\vert _{g_{J}}^{2}=d\alpha\left(J\frac{du(z)}{dx},\frac{du(z)}{dx}\right)+\nu'ds\wedge\alpha\left(J\frac{du(z)}{dx},\frac{du(z)}{dx}\right)\text{ for }z\in a^{-1}(a_{0}-\frac{\delta}{2},a_{0}+\frac{\delta}{2})
\end{equation}
(or in other words that we can choose the taming constant $C_{T}=1$
in the notation of \cite[Lemma 2.5]{CieliebakMohnke2005}). By summing
these inequalities over the different components and using the definition
of the Hofer energy we see that 
\begin{equation}
\frac{C\delta^2}{4}\cdot\#\{\text{connected components in }Z_{R}^{S}(u),Z_{S}^{R}(u)\}\leq2E(u)
\end{equation}

\end{proof}
Following \cite{CieliebakMohnke2005} we call a subset $P_{0}\subset Z$
a $\delta$-\emph{essential local minimum on level $R_{0}$} of $u:Z\rightarrow\mathbb{R}\times\Sigma$
if $P_{0}$ is a connected component of $a^{-1}((-\infty,R_{0}+\delta])$
and $R_{0}=\min_{P_{0}}a$. Similarly, a $\delta$-\emph{essential
local maximum $P_{0}$ on level $R_{0}$} is a connected component
$P_{0}$ of $a^{-1}([R_{0}-\delta,\infty))$ with $R_{0}=\max_{P_{0}}a$.
Observe that any two different $\delta$-essential local minima are disjoint.
\begin{lem}
\cite[Lemma 5.6]{CieliebakMohnke2005}\label{lem:extremabound} For
any $J$-holomorphic curve $u:Z\rightarrow\mathbb{R}\times\Sigma$
with $u(\partial Z)\subset(K+4\delta,\infty)\times\Sigma$ the number
of $\delta$-essential local minima of $u$ in $(-\infty,K)\times\Sigma$ is
bounded above by $2E(u)/C\delta^{2}$. Furthermore there are
no $\delta$-essential local maxima in $(-\infty,K)\times\Sigma$. \end{lem}
\begin{proof}
Let $P_{i},\, i=1,\ldots,p$ denote the $\delta$-essential local
minima of $u$ with critical values $R_{i}$. By definition, the $P_{i}$
are pairwise disjoint and therefore 
\begin{equation}
\int_{Z}u^{*}d\alpha\geq\sum_{i=1}^{p}\int_{P_{i}}u^{*}d\alpha
\end{equation}
On the other hand if we choose a function $\nu\in\mathcal{S}$ with
$\nu'(s)=1$ for $s\in[R_{i},R_{i}+\delta]$ and $\nu(R_{i}+\delta)=1$
we see that 
\begin{equation}
2\int_{P_{i}}u^{*}d\alpha\geq\int_{P_{i}}u^{*}d\alpha+\int_{P_{i}}u^{*}(\nu'ds\wedge\alpha)\geq\mbox{area}_{g_{J}}(u\vert_{a^{-1}((R_{i}-\delta,R_{i}+\delta))\cap P_{i}})\geq C\delta^{2}
\end{equation}
where we use first Stokes' Theorem, and then the compatibility of $J$ with
$g_{J}$ and the monotonicity Lemma \ref{lem:monotonicity} as above.
Summing over the different local minima the claimed bound follows
immediately by the definition of $E(u)$. It is well known that a
nonconstant pseudoholomorphic curve in the symplectization $\mathbb{R}\times\Sigma$
does not have any (interior) local maxima by the maximum principle,
therefore in particular no $\delta$-essential local maxima in $(-\infty,K)\times\Sigma$.
This can also be seen by choosing $\nu\in\mathcal{S}$ with $\nu'(s)=1$
in $[R_{0}-\delta,R_{0}]$ and $\nu(R_{0}-\delta)=0$ for a local
maximal value $R_{0}$. The computation above implies that the curve
has no area near the local maximum and therefore the map is constant
on this connected component, a contradiction to our assumption that
$u$ has no constant components. \end{proof}
\begin{lem}
\cite[Lemma 5.7]{CieliebakMohnke2005}\label{lem:Eulercharbound}
Let $u=(a,f)$ be a pseudoholomorphic curve of total genus $\leq g$
and denote by $A:=a^{-1}((-\infty,b-\delta])\cup\mathcal{C}_{b}^{-}$.
If we make the assumptions as above, then we have 
\begin{equation}
\chi(A)\geq2-3g-12E(u)/C\delta^{2}
\end{equation}
and the following estimates on the Euler characteristic for $R<S<K$:
\begin{eqnarray}
\chi(A)-8E(u)/C\delta^{2}\leq\chi(Z_{R}^{S}(u))\leq8E(u)/C\delta^{2}\\
\chi(A)-8E(u)/C\delta^{2}\leq\chi(Z_{S}^{R}(u))\leq8E(u)/C\delta^{2}
\end{eqnarray}
\end{lem}
\begin{proof}
Using the same procedure as in Lemma \ref{lem:componentbound}, we
first observe that the number $p$ of components of $A$ is at most
$8E(u)/C\delta^{2}$. In the next step we want to get an upper bound
on the number of $q$ of boundary components of $A$ (on level $b$),
and therefore a lower bound on the Euler characteristic of $A$. In
order to do that we look at how the boundary components of $A$ on
level $b$ behave, when we extend our domain of consideration to the
extended surface $\tilde{A}:=a^{-1}((-\infty,b])\cup\mathcal{C}_{b+\delta}^{-}$.
Since there are no local maxima, there are only two options. Some
of these boundary components will connect with each other in $a^{-1}(b,b+\delta)$,
and some of them will stay apart. Since the genus is at most $g$,
and any connection of two boundary components which lie in the same
component of $A$ will increase the genus or reduce the number of connected
components of (the extended) $A$, there are at most $p+g$ new connections
which join the different boundary components of $A$. Let us call
two boundary components of $A$ \emph{equivalent} if they lie in the same
connected component of $\tilde{A}$. By the above discussion, there
are at most $q-(p+g)$ equivalence classes of boundary components of $A$, each of them representing a component of $\mathcal{C}_{b+\delta}^{-}$
which intersects both the levels $b$ and $b+\delta$. By the monotonicity
Lemma \ref{lem:monotonicity} we see that there can be at most $4E(u)/C\delta^{2}$ such components,
i.e. $q\leq p+g+4E(u)/C\delta^{2}$. Finally since $\chi(A)\geq2-2g-q$
we obtain the first inequality. By Lemma \ref{lem:componentbound},
both $Z_{R}^{S}(u)$ and $Z_{S}^{R}(u)$ have at most $8E(u)/C\delta^{2}$
components. Furthermore, all of these components have Euler characteristic
at most 1, therefore the upper bounds follow. Next, since the Euler
characteristic is additive under gluing along common boundary, we
see that $\chi(A)=\chi(Z_{S}^{R}(u))+\chi(Z_{R}^{S}(u))$. From the
lower bounds on $\chi(A)$ and the upper bounds on each of the terms
on the right hand side, we find the claimed lower bounds for each
of the terms on the right hand side. \end{proof}
\begin{lem}
\cite[Lemma 5.8]{CieliebakMohnke2005}\label{lem:boundminima} With
the notation from above, 
\begin{equation}
\chi(Z_{R}^{S}(u))\leq\#\{\delta\text{-essential local minima }Z_{0}\subset Z_{R}^{S}(u)\}.
\end{equation}
If $Z_{R}^{S}(u)$ contains no $\delta$-essential local minima and $\chi(Z_{R}^{S}(u))=0$,
then $Z_{R}^{S}(u)$ is the disjoint union of cylinders connecting
levels $R$ and $S$. \end{lem}
\begin{proof}
If a component in $Z_{R}^{S}(u)$ has positive Euler characteristic,
then it is necessarily a disk with boundary on level $S$ since no
local maxima can occur. This disk will (by definition of $Z_{R}^{S}(u)$)
contribute a $\delta$-essential local minimum. If there are no $\delta$-essential
minima, then by the argument above, all components of $Z_{R}^{S}(u)$
have nonpositive Euler characteristic. If $\chi(Z_{R}^{S}(u))=0$,
then this implies that the Euler characteristic of all components
is equal to $0$, and that they are therefore cylinders connecting
the $R$ and $S$ levels, since there are no $\delta$-essential extremas. 
\end{proof}
We introduce now the function $\chi_{u}:(-\infty,K]_{\text{reg}}\rightarrow\mathbb{Z}$
on the set of regular values of $a$ by 
\begin{equation}
\chi_{u}(r):=\chi(Z_{r}^{b}(u))
\end{equation}
A value $r\in(-\infty,K]$ is called an \emph{upward jump}, if 
\begin{equation}
\limsup_{S\searrow r}\chi_{u}(S)-\liminf_{R\nearrow r}\chi_{u}(R)>0,
\end{equation}
where the limits are taken over regular values $S$ and $R$. Similarly,
one can define a \emph{downward jump}.
\begin{lem}
\cite[Lemma 5.9]{CieliebakMohnke2005}\label{lem:boundjumps} The
number of downward jumps of $\chi_{u}$ is at most $E(u)/C\delta^{2}$.
The number of all upward jumps of $\chi_{u}$ is at most $3g+5E(u)/C\delta^{2}$. \end{lem}
\begin{proof}
Consider regular values $R<S.$ The difference $\chi_{u}(R)-\chi_{u}(S)$
is the Euler characteristic of $Z_{R}^{b}(u)\setminus Z_{S}^{b}(u)$.
Any of the components of this surface has its boundary components
on level $R$ or $S$. It contributes some positive Euler characteristic
exactly if it is a disk $C$ with boundary component on level $S$.
By construction of $Z_{R}^{b}(u)$ and $Z_{S}^{b}(u)$, we know that $R-\delta<\min_{C}a<S-\delta$
(otherwise it would either not belong to $Z_{R}^{b}(u)$ or it would
belong to $Z_{R}^{b}(u)$). This implies that $C$ contains an $\delta$-essential
local minimum in the interval $[R-\delta,S-\delta]$. Choosing $R\nearrow r$
and $S\searrow r$, we find that at any downward jump $r\leq K$,
there must be a $\delta$-essential local minimum on level $r-\delta$. Now
the first claim follows from Lemma \ref{lem:extremabound}. Furthermore,
from the estimate in Lemma \ref{lem:Eulercharbound}, we know that
the Euler characteristic of $Z_{\mathrm{min\,}u-1}^{S}(u)$ lies always in
a interval of length at most $3g+28E(u)/C\delta^{2}$, therefore,
since $Z_{\mathrm{min\,}u-1}^{S}(u)\cup Z_{S}^{b}(u)=Z_{\mathrm{min\,}u-1}^{b}(u)$,
so does $Z_{S}^{b}(u)$. Using our previous bound on the number of
upward jumps, the result follows. 
\end{proof}
In conclusion the bounded function $\chi_{u}$ extends to a locally
constant function (also denoted by) $\chi_{u}:(-\infty,K]\rightarrow\mathbb{Z}$ with
finitely many jumps. We summarize the results of Lemmas \ref{lem:componentbound} - \ref{lem:boundjumps}
in the following Proposition.
\begin{prop}
\label{prop:allbounds} Let $E>0,\delta>0$ and $g\in\mathbb{N}$
and $K\in\mathbb{R}$. Then there exists an $N_{0}\in\mathbb{N}$
such that for any pseudoholomorphic curve $u=(a,f):Z\rightarrow\mathbb{R}\times\Sigma$
which: 
\begin{enumerate}
\item is defined on a compact manifold $Z$ of genus at most $g$, 
\item has Hofer energy $E(u)\leq E$,
\item is nonconstant on all components of $Z$ and satisfies $u(\partial Z)\subset(K+4\delta,\infty)\times\Sigma$ 
\end{enumerate}
there are at most $N_{0}$ jumps of $\chi_{u}:(-\infty,K]\rightarrow\mathbb{Z}$
and the number of $\delta$-essential minima of $u$ below level $K$
is bounded by $N_{0}$. Furthermore, if $R<S<K,R\pm\delta,S\pm\delta$
are all regular values of $a$, $\chi(Z_{R}^{S}(u))=0$, and there are
no $\delta$-essential minimas in $Z_{R}^{S}(u)$, then $Z_{R}^{S}(u)$
is a union of at most $N_{0}$ cylinders connecting the level $R$
with the level $S$. \end{prop}
\begin{proof}
The bound on the number of jumps comes from Lemma \ref{lem:boundjumps},
the bound on the number of $\delta$-essential local minima below level $K$ follows
from Lemma \ref{lem:boundminima}. In the last statement all components
are cylinders connecting level $R$ with level $S$ again by Lemma
\ref{lem:boundminima}, and the bound on the number of such cylinders
follows immediately from Lemma \ref{lem:componentbound}. 
\end{proof}

\begin{rem}\label{rem:modulus}
Recall that any compact Riemann surface $C$ which
is diffeomorphic to a (finite) cylinder is biholomorphically equivalent
to a standard cylinder $([0,L]\times \R/\Z,i)=[0,L]\times S^1$ for a uniquely determined $L$. The number $L$ is called \emph{the conformal modulus} of $C$ and $\sqrt{L^{-1}}$ is called \emph{the conformal length} of $C$ (see for instance \cite{Ahlfors1973} for more information).
\end{rem}

In order to prove Proposition \ref{prop:finding_subcylinder}, we need
to bound the conformal modulus of a pseudoholomorphic cylinder from below.
This is achieved by the next lemma. 
\begin{lem}
\cite[Lemma 4.20-4.22]{CieliebakMohnke2005}\label{lem:modulusbound}
Let $u=(a,f):[0,L]\times S^{1}\rightarrow\mathbb{R}\times\Sigma$
be a pseudoholomorphic cylinder with $a(0,t)\leq R$ and $a(L,t)\geq S$
for all $t\in S^{1}$ and some $R<S$. Then the conformal modulus $L$ of the
cylinder is bounded below by $(S-R)/2E(u)$. \end{lem}
\begin{proof}
The proof we outline is based on a method to bound the \emph{conformal
length} of a family $\Gamma$ of curves from above, which can be found
in \cite[p.55-56]{Gromov1983}. For each $s\in[R,S]$ which is a regular
value of $a$, denote by $\gamma_{s}:=a^{-1}(s)$ its preimage, consisting
of circles which separate $[0,L]\times S^{1}$. Now let $\Gamma=\{\gamma_{s}\mid R<s<S\text{ is a regular value of }a\}$
be the collection of those curves and let $g_{0}$ be the (possibly
singular) metric on $[0,L]\times S^{1}$ induced by $u$. The conformal
length of $\Gamma$ is by definition 
\begin{equation}
\text{conf length}(\Gamma):=\sup_{g}\inf_{s}|\gamma_{s}|_{g}
\end{equation}
where $g$ runs over all conformal metrics $g=\varphi^{2}g_{0}$ on
$[0,L]\times S^{1}$ with area $\int_{[0,L]\times S^{1}}\varphi^{2}d\mbox{vol}_{g_{0}}\leq1$.
Fix now such a $g$ and compute 
\begin{equation}
\inf_{s}|\gamma_{s}|_{g}\leq\tfrac{1}{S-R}\int_{R}^{S}|\gamma_{s}|_{g}ds=\tfrac{1}{S-R}\int_{R}^{S}\int_{\gamma_{s}}\varphi(v)d\gamma_{s}(v)ds=\tfrac{1}{S-R}\int_{a^{-1}([R,S])}\varphi|\nabla a|_{g_{0}}d\mbox{vol}_{g_{0}}.
\end{equation}
Here the last step follows from the coarea formula and the fact that $\gamma_{s}=a^{-1}(s)\subset[0,L]\times S^{1}$
is endowed with the metric induced by $g_{0}$. Since $|\nabla a|_{g_{0}}\leq1$, $\mbox{area}_{g}([0,L]\times S^{1})\leq1$, and $\mbox{area}_{g_{0}}(a^{-1}([s,s+1]))\leq 2E(u)$, the Cauchy-Schwarz inequality and a computation as in equation \eqref{eqn:monotonicity_application} implies that
\begin{equation}
\inf_{s}|\gamma_{s}|_{g}\leq\tfrac{1}{S-R}\sqrt{\mbox{area}_{g_{0}}(a^{-1}([R,S]))}\leq\tfrac{1}{S-R}\sqrt{(S-R)2E(u)}
\end{equation}
On the other hand, if the metric is nonsingular, then a special choice of metric 
is given by $g_{L}=L^{-1}g_{\text{Euclidean}}$, and in this case 
\begin{equation}
\inf_{s}|\gamma_{s}|_{g_{L}}\geq L^{-1/2}
\end{equation}
Combining these inequalities gives the claim in the case where $g_{0}$
is nonsingular. By concentrating on the full measure, open subset
of $s\in[R,S]$ of regular values of $a$, a suitable subdivision
of the interval $[R,S]$ makes the above proof also applicable in
the singular case. Indeed, the cylinder contains disjoint open subcylinders
of the form $a^{-1}((r_{i},s_{i}))$ with $r_{i+1}=s_{i}$ and $[R,S]=\bigcup_{i}[r_{i},s_{i}]$
with nonsingular induced metrics. Therefore we have, as above, bounds
on the conformal moduli $L_{i}$ of these subcylinders. Summing up these
lower bounds gives the claimed lower bound for the conformal modulus
$L$ of the cylinder. \end{proof}

We finally come to the proof of Proposition \ref{prop:finding_subcylinder}.
\begin{prop}
\label{prop:finding_subcylinder2} 
Suppose $(Z_{k},j_{k})$ is a family of compact (possibly disconnected) Riemann surfaces with boundary and uniformly bounded genus. Assume that 
\begin{equation}
u_{k}=(a_{k},f_{k}):Z_{k}\rightarrow \R\times\Sigma
\end{equation}
s a sequence of $(j_{k},J)$-holomorphic maps which have uniformly bounded Hofer energy $E(u_{k})\leq E$, are nonconstant on each connected
component of $Z_k$, and satisfy $a_{k}(\partial Z_{k})\subset [0,\infty)$. \\
Assume that $\inf_k \inf_{Z_k} a_k=-\infty$.
Then there exists a subsequence $k_n$
and cylinders $C_n\subset Z_{k_n}$ such that an $\mathbb{R}$-shift
$v_{n}=(b_n,g_n)$ of the restriction of $u_{k_n}$ to $C_{n}$ has the following
properties: 
\begin{enumerate}
\item $C_{n}$ is biholomorphic to $[-L_{n},L_{n}]\times S^{1}$ and $L_{n}\rightarrow\infty$
as $n\rightarrow\infty$ 
\item $\int_{C_n} v_{n}^{*}d\alpha\rightarrow0$ 
\item After a suitable identification of $C_n$ with $[-L_{n},L_{n}]\times S^{1}$ there is a sequence $\sigma_{n}\rightarrow\infty$ such that $\pm b_{n}(\pm L_{n},t)\geq\sigma_{n}$
for each $t\in S^{1}$. 
\end{enumerate}
\end{prop}
\begin{proof}
We follow now \cite[Section 5.4]{CieliebakMohnke2005} and call
a level $r\in(-\infty,K]$ \emph{essential for $u_{k}$}
if it satisfies one of the following conditions: 
\begin{itemize}
\item $r=\mbox{min\,}u_{k}$ or $r=K$, 
\item $u_{k}$ has a $\delta$-essential minimum on level $r-\delta$, 
\item $\chi_{k}$ has a jump at level $r$. 
\end{itemize}
Observe that the $u_k$'s satisfy the assumptions of Proposition \ref{prop:allbounds} for $K=-5\delta$.
Therefore the number of essential levels is
bounded independently of $k$ by Proposition \ref{prop:allbounds}. Since $\inf\min_{Z_k} u_k= -\infty$,
this implies in particular that for a subsequence $k_n$, there are intervals 
$[\rho_n,\sigma_n]\subset [\mbox{min\,}u_{k_n},K]$ with $\sigma_n-\rho_n\geq n^2$ which do not contain any critical level. 
Next observe that if we cut the interval $[\rho_n,\sigma_n]$ into $n$ pieces $I_n^j$ of equal length, then for at least one $j=j(n)\in\{1,\ldots ,n\}$, 
\begin{equation}
\int_{a_{k_n}^{-1}(I_n^{j(n)})}u_{k_n}^*d\alpha\leq \frac{1}{n}\int_{Z_{k_n}} u_{k_n}^*d\alpha\leq E/n \text{, and length}(I_n^{j(n)})\geq n.
\end{equation}
Subdivide $I_n^{j(n)}$ into 3 subintervals of equal length and denote
the middle one by $I_n:=[\rho_n',\sigma_n']$. 
By possibly shrinking $I_n$ slightly, but keeping the notation the same,
we may assume that $\rho_n',\sigma_n',\rho_n'\pm\delta,\sigma_n'\pm\delta$ are all regular values of $a_{k_n}$ and 
(in the notation of \eqref{eqn:decomposition}) $Z_{\rho_n'}^{\sigma_n'}(u_{k_n})\subset a_{k_n}^{-1}([\rho_n'-\delta,\sigma_n'+\delta])\subset a_{k_n}^{-1}(I_n^{j(n)})$ if $n\geq 3$ and hence
\begin{equation}
\int_{a_{k_n}^{-1}(I_n)}u_{k_n}^*d\alpha\leq E/n \text{, and length}(I_n)\geq n/4.
\end{equation}
We infer from the last statement in Proposition \ref{prop:allbounds} that  $a_{k_n}^{-1}(I_n)$ is
parametrized by a disjoint union of cylinders (which is necessarily nonempty,
since by the maximum principle each component of $u_{k}$ meets the level $a_{k}^{-1}(K)$).\\
For each $n$, choose a component $C_n\subset a_{k_n}^{-1}(I_n)$. Then
$a_{k_n}(\partial C_n)= \{\rho_n',\sigma_n'\}$ and so we derive from 
Lemma \ref{lem:modulusbound} that the conformal modulus of $C_n$
tends to infinity. It is now immediate that the restriction of $u_{k_n}$ to $C_n$ satisfies 
(1) and (2). Finally if we set $b_{k_n}:=a_{k_n}+\rho_n'+(\sigma_n'-\rho_n')/2$,
then the $\R$-shift $v_n:=(b_n,f_{k_n})$ of $u_{k_n}$ restricted to $C_n$ in addition satisfies (3), after possibly switching the boundary components of $C_n$ by a biholomorphic map.
\end{proof}

\bibliographystyle{amsalpha}
\bibliography{willmacbibtex}

\end{document}